\documentclass[11pt]{article}

\usepackage{pgf,acadpgf}
\usepackage{amssymb}
\usepackage{amsmath}
\usepackage{amsthm}
\usepackage{tikz-cd}

\numberwithin{equation}{section}

\newtheoremstyle{slplain}
  {\topsep}
  {\topsep}
  {\slshape}
  {0pt}
  {\bfseries}
  {.}
  {0.5em}
  {}

\theoremstyle{slplain}
  \newtheorem{THM}{Theorem}[section]
  \newtheorem{LEM}[THM]{Lemma}
  
  \newtheorem{COR}[THM]{Corollary}
  \newtheorem*{THMNONUM}{Theorem}

\theoremstyle{definition}
  \newtheorem{DEF}[THM]{Definition}
  \newtheorem{EX}{Example}[section]
  
  \newtheorem{ZAD}[THM]{Exercise}
  \newtheorem{PROB}[THM]{Problem}

\newcommand{\Sax}{\Sigma_{ax}}
\newcommand{\Dax}{\Delta_{ax}}
\newcommand{\XRow}{X_r}
\newcommand{\XCol}{X_c}
\newcommand{\Cns}{\mathrm{Prop}}

\newcommand{\id}{\mathrm{id}}
\newcommand{\DedRuleII}[4]{\begin{array}{lcr}
  #1 & & #2\\ \hline \hbox to #4{\hfill}& #3 &\hbox to #4{\hfill}
\end{array}}
\newcommand{\Uniq}{\mathit{Uniq}}
\newcommand{\Th}{\mathrm{Th}}

\renewcommand{\le}{\leqslant}
\renewcommand{\ge}{\geqslant}

\newcommand{\0}{\varnothing}

\renewcommand{\phi}{\varphi}
\renewcommand{\epsilon}{\varepsilon}

\newcommand{\union}{\cup}

\newcommand{\calA}{\mathcal{A}}
\newcommand{\calB}{\mathcal{B}}

\newcommand{\calL}{\mathcal{L}}

\newcommand{\calN}{\mathcal{N}}

\newcommand{\calS}{\mathcal{S}}

\title{Deducibility in Sudoku}
\author{%
  Dragan Ma\v sulovi\'c\\
  University of Novi Sad, Faculty of Sciences\\
  Department of Mathematics and Informatics\\
  Trg Dositeja Obradovi\'ca 3, 21000 Novi Sad, Serbia\\
  e-mail: dragan.masulovic@dmi.uns.ac.rs}

\begin{document}
\maketitle

\begin{abstract}
  In this paper we provide a formalism, \emph{Sudoku logic}, in which a solution is logically deducible
  if for every cell of the grid we can \emph{provably} exclude all but a single option.
  We prove that the deductive system of Sudoku logic is sound and complete, and as a consequence of that
  we prove that a Sudoku puzzle has a unique solution if and only if it has a deducible solution.
  Using the classification of fundamental Sudoku transformations by Adler and Adler
  we then formalize the notion of symmetry in Sudoku and provide a formal proof of Gurth's Symmetrical
  Placement Theorem. In the concluding section of the paper we present a Sudoku formula
  that captures the idea of a Sudoku grid having a unique solution. It turns out that this formula
  is an axiom of Sudoku logic, making it possible for us to offer to the Sudoku community a
  resolution of the Uniqueness Controversy: if we accept Sudoku logic as presented in this paper,
  there is no controversy! Uniqueness is an axiom and, as any other axiom, may freely
  be used in any Sudoku deduction.
\end{abstract}

\section{Introduction}

Sudoku is a logic game for one player where
the goal is to fill the $9\times9$ grid with digits 1, 2, \ldots, 9
so that in every row, every column and every $3\times3$ box each of the digits 1, 2, \ldots, 9
appears exactly once.

Usually the initial setting contains some given digits, and the player is expected to fill in the rest.
What is the smallest number of given digits that ensures a unique solution of a Sudoku puzzle?
It had been conjectured that the answer is 17, and this was confirmed by an exhaustive computer search
in~2014 by McGuire, Tugemann and Civario~\cite{no16cluesudoku}.

Sudoku is traditionally treated by mathematicians as a combinatorial topic.
For example, Felgenhauer and Jarvis proved in~\cite{FelgenhauerJarvis} that there are
$$
  6\,670\,903\,752\,021\,072\,936\,960 \approx 6.67 \cdot 10^{21}
$$
different fully filled Sudoku grids. Now, not all of them are \emph{essentially different}: if we take a fully filled
Sudoku grid and rotate it clockwise by $90^\circ$ we end up with a new Sudoku grid,
but the two can hardly be considered ``essentially different''. When it comes to counting essentially different Sudoku grids, 
Russell and Jarvis proved in~\cite{RussellJarvis} that only $5\,472\,730\,538$ out of $6.67 \cdot 10^{21}$ fully filled
Sudoku grids are essentially different.\footnote{We shall define precisely in Section~\ref{sudoku.sec.transf} what does it mean for Sudoku grids to be essentially different.}

Solving Sudoku puzzles by trial an error is considered unacceptable by Sudoku puritans:
instead we are interested in solutions that are \emph{logically deducible} from the initial
setup. However, the notion of logical deducibility in Sudoku has not been formally specified, so
in this paper we intend to rectify this injustice.

Unlike most mathematical treatments of Sudoku which approach the problems from the combinatorial point of view
(see e.g.~\cite{SudokuSeriously}), in this paper we focus on the notion of deducibility.
The main contribution of this paper is \emph{Sudoku logic}, a formal deductive system
in which a solution is logically deducible
if for every cell of the grid we can \emph{provably} exclude all but a single option.
This point of view is a clear display of the general feeling that Sudoku is a game of elimination.

The beginning of the paper is devoted to the exposition of Sudoku logic:
Section~\ref{sudoku.sec.Formalism} contains a brief discussion of the general nature of
formal deductive systems, Section~\ref{sudoku.sec.Syntax} defines the syntax,
while Section~\ref{sudoku.sec.Sem} defines the semantics of Sudoku logic.

The first nontrivial result of the paper appears in Section~\ref{sudoku.sec.Deduc}
where we introduce the notion of deducibility. We prove there that the deductive system
of Sudoku logic is sound an complete (although the rather technical proof of completeness is
deferred until Appendix~\ref{sudoku.sec.app-A}), and as the immediate consequence of
soundness and completeness we prove that a Sudoku puzzle has a unique solution \emph{if and only if}
it has a logically deducible solution.\footnote{Statements of the form \emph{if and only if} are not typical for
the spoken language, although they are ubiquitous in formal mathematical texts; for example,
this statement means that: (1) if a Sudoku grid has a unique solution then this solution is logically deducible,
and (2) if a Sudoku grid has a logically deducible solution then it is the only solution.}
In other words,
\begin{center}
  uniqueness $=$ logical deducibility.
\end{center}

To demonstrate the strength of Sudoku logic, we devote the remaining part of the paper
to a precise formulation and the proof of Gurth's Symmetrical Placement Theorem, which,
informally, takes the following form:
\begin{quote}
  \textbf{Gurth's Symmetrical Placement Theorem (informally).}
  \textsl{If the initial setup of a Sudoku puzzle exhibits some kind of symmetry,
  and we know that the puzzle has a unique solution, then the solution to the puzzle
  exhibits the same kind of symmetry.}
\end{quote}
But, what is a symmetry? To provide a precise answer, in Section~\ref{sudoku.sec.transf}
we provide a brief discussion of the general nature of symmetry, and the relationship between
a ``symmetrical setup'' and transformations that preserve the setup. 
Starting from the classification of fundamental transformations by Adler and Adler~\cite{AdlerAdler}
we derive in Section~\ref{sudoku.sec.transf-def} the notion of Sudoku transformation which
ina standard manner leads to the notion of automorphism of a Sudoku grid.
The notions of logical deducibility and Sudoku automorphism then make it possible for
us to formally define and prove Gurth's Symmetrical Placement Theorem in Section~\ref{sudoku.sec.Gurth}.

We close the paper with Section~\ref{sudoku.sec.conclusion} in which we tackle Uniqueness Controversy in Sudoku.
We present a formula of Sudoku logic which expresses the fact that a Sudoku grid has a unique solution,
and then show (rather easily) that this is one of the axioms of Sudoku logic. Therefore,
\begin{quote}
  \textsl{if we accept Sudoku logic as presented in this paper, uniqueness of the solution is just another axiom
  of the deductive system; since axioms are always at our disposal in any formal deduction,
  it follows that there is nothing controversial in assuming uniqueness while searching for a logically deducible solution
  for a Sudoku grid!}
\end{quote}

This paper is written under the assumption that the reader is familiar with basics of Sudoku,
as well as with basics of mathematics at the level of any standard discrete mathematics course
for undergraduate computer science students (see for example~\cite{dm1}). Most of the notions
have been introduced slowly and with ample motivation, except for Appendix~\ref{sudoku.sec.app-A}
which is intended for more mathematically inclined readers.

\section{Formal logic}
\label{sudoku.sec.Formalism}

Formal logic (or a formal deductive system)
is a set of formal rules that enables us, humans, to describe and talk competently about
logical reasoning and logical deductions within a ``universe''.
Nowadays there are many different formal logics, but for our purposes
the simplest one, \emph{propositional logic}, will suffice.

The \emph{Sudoku logic} that we now
introduce is a very special propositional logic. Just as a refresher, let us recall that
in propositional logic statements about the state of affairs in the ``universe'' are written
using \emph{propositional formulas} which are formed from \emph{simple propositions} (such as $p$, $q$, $r$, $s$, \ldots),
logical connectives ($\lnot$, $\land$, $\lor$, $\Rightarrow$ and $\Leftrightarrow$) and parentheses like so:
$$
  (\lnot p \lor q) \Rightarrow ((p \land s) \Leftrightarrow \lnot(q \lor s)).
$$
(See~\cite{dm1} for a lot of motivation, and~\cite[Chapter~1]{hedman} for a standard
presentation of the classical propositional logic.) 

Sudoku logic is tailor-made to formally reason about a universe that, in our case, is the classic Sudoku grid.
A Sudoku grid is organized into nine rows ($r_1$, $r_2$, \ldots, $r_9$),
nine columns ($c_1$, $c_2$, \ldots, $c_9$), and nine boxes,~Fig.~\ref{sudoku.fig.notation}.
A cell in row $i$ and column $j$ will be denoted by $r_i c_j$.
Cells $r_i c_i$, $1 \le i \le 9$, form the \emph{main diagonal} of the grid,
while cells $r_i c_{10-i}$, $1 \le i \le 9$, form the \emph{auxiliary diagonal} of the grid~Fig.~\ref{sudoku.fig.notation}.

\begin{figure}
  \centering
  \input pic/notation.pgf
  \caption{Sudoku notation}
  \label{sudoku.fig.notation}
\end{figure}

Simplest statements that we can make about the state of the affairs
are \emph{Sudoku propositions} which take the following form:
$$
  r_ic_j \not\approx d
$$
meaning: in the grid, the cell in row $i$ column $j$ \emph{cannot possibly} contain digit $d$.
Simple propositions then combine into formulas just like in the classical propositional logic.
For example:
$$
  (\lnot(r_1c_1 \not\approx 1) \land \lnot(r_1c_2 \not\approx 3)) \Rightarrow (r_1c_3 \not\approx 1 \land r_1c_3 \not\approx 3).
$$
\begin{center}
  \def\h#1{{\large\textbf{#1}}}
\begin{pgfpicture}
  \pgfsetxvec{\pgfpoint{\acadpgfunit}{0pt}}
  \pgfsetyvec{\pgfpoint{0pt}{\acadpgfunit}}
  \pgfsetlinewidth{\acadpgflinewidth}
  \pgftransformshift{\pgfpointxy{-37.5}{-12.5}}

  \begin{pgfscope}
    \pgfsetlinewidth{0.50mm}
    \pgfpathmoveto{\pgfpointxy{150.0}{100.0}}
    \pgfpathlineto{\pgfpointxy{150.0}{250.0}}
    \pgfusepath{stroke}
  \end{pgfscope}
  \begin{pgfscope}
    \pgfsetlinewidth{0.50mm}
    \pgfpathmoveto{\pgfpointxy{150.0}{250.0}}
    \pgfpathlineto{\pgfpointxy{600.0}{250.0}}
    \pgfusepath{stroke}
  \end{pgfscope}
  \begin{pgfscope}
    \pgfpathmoveto{\pgfpointxy{250.0}{100.0}}
    \pgfpathlineto{\pgfpointxy{250.0}{250.0}}
    \pgfusepath{stroke}
  \end{pgfscope}
  \begin{pgfscope}
    \pgfpathmoveto{\pgfpointxy{350.0}{100.0}}
    \pgfpathlineto{\pgfpointxy{350.0}{250.0}}
    \pgfusepath{stroke}
  \end{pgfscope}
  \begin{pgfscope}
    \pgfsetlinewidth{0.50mm}
    \pgfpathmoveto{\pgfpointxy{450.0}{100.0}}
    \pgfpathlineto{\pgfpointxy{450.0}{250.0}}
    \pgfusepath{stroke}
  \end{pgfscope}
  \begin{pgfscope}
    \pgfpathmoveto{\pgfpointxy{550.0}{100.0}}
    \pgfpathlineto{\pgfpointxy{550.0}{250.0}}
    \pgfusepath{stroke}
  \end{pgfscope}
  \begin{pgfscope}
    \pgfpathmoveto{\pgfpointxy{150.0}{150.0}}
    \pgfpathlineto{\pgfpointxy{600.0}{150.0}}
    \pgfusepath{stroke}
  \end{pgfscope}
  \pgftext[at={\pgfpointxy{200.0}{200.0}}]{\h1}
  \pgftext[at={\pgfpointxy{300.0}{200.0}}]{\h3}
  \pgftext[bottom,at={\pgfpointxy{200.0}{262.0}}]{$c_1$}
  \pgftext[bottom,at={\pgfpointxy{300.0}{262.0}}]{$c_2$}
  \pgftext[bottom,at={\pgfpointxy{400.0}{262.0}}]{$c_3$}
  \pgftext[right,at={\pgfpointxy{138.0}{200.0}}]{$r_1$}
\end{pgfpicture}
\end{center}

The purpose of the following three sections is to define Sudoku logic formally.
Defining a formal logic (or a formal deductive system) goes in three steps.
\begin{description}
  \item[Step 1.]
    We first have to define the \emph{syntax} of the system, which tells us what correctly formed strings of
    symbols look like. For example, we will ensure that
    $$
      ((\lnot(r_1c_1 \not\approx 1) \land \lnot(r_1c_2 \not\approx 3)) \Rightarrow (r_1c_3 \not\approx 1 \land r_1c_3 \not\approx 3))
    $$
    is a correct string of symbols (don't worry about the outer parentheses for now, their presence simplifies
    formal reasoning, but we shall very quickly relax the rigid rules imposed by the definition), and that
    $$
        r_1c_1 \lor \land (r_2c_2 = 5) \Rightarrow (
    $$
    is not. Then we ignore the incorrectly formed strings and call the correct ones \emph{formulas}.
  \item[Step 2.]
    We then have to define the \emph{semantics} of the system, which provides \emph{meaning} to formulas
    (correctly formed strings). This is usually straightforward, as we shall se below.
  \item[Step 3.]
    Finally, we have to describe how \emph{deduction} in the system works.
    Formal deductive systems are required to be \emph{sound}. Fortunately, the soundness of our Sudoku logic
    comes from the fact that it is a special propositional logic, and propositional logic is known to be
    sound (see~\cite{hedman} for details).
\end{description}

\section{Sudoku logic: Syntax}
\label{sudoku.sec.Syntax}

The syntax of Sudoku logic is given by a pair of definitions below. We first define \emph{Sudoku propositions}, which are the
basic building blocks of Sudoku formulas. We can then combine them into more elaborate statements (\emph{Sudoku formulas}) 
using the classic logical connectives (and a lot of parentheses to ensure the unique parsing of the formula).

Sudoku is often characterized as a game of elimination. We shall now formally support this claim by
choosing negative constraints as building blocks of Sudoku formulas.

\begin{DEF}
  A \emph{Sudoku proposition} is a string of symbols of the form
  $$
    r_ic_j \not\approx d,
  $$
  where $i$, $j$ and $d$ can be any of the numbers 1, 2, \ldots, 9.
\end{DEF}

\begin{DEF}
  A \emph{Sudoku formula} is a string of symbols formed according to the following rules:
  \begin{itemize}
    \item \emph{Sudoku propositions} and the symbols $\top$ and $\bot$ are Sudoku formulas;
    \item if $\alpha$ and $\beta$ are Sudoku formulas, then so are the following strings of symbols:
          $\lnot \alpha$, $(\alpha \land \beta)$, $(\alpha \lor \beta)$,
          $(\alpha \Rightarrow \beta)$ and $(\alpha \Leftrightarrow \beta)$;
    \item every Sudoku formula can be constructed by applying the above two rules finitely many times.
  \end{itemize}
\end{DEF}

The first two items in this definition tell us how to build Sudoku formulas:
start from Sudoku propositions and symbols $\top$ and $\bot$,
and then either take a formula we have already built and prefix it with $\lnot$,
or take two formulas we have already built, connect them by $\land$, $\lor$, $\Rightarrow$ or $\Leftrightarrow$
and put parentheses around.

The last item is subtle: it tells us that every Sudoku formula is built this way (that is, other strings of symbols
are \emph{not} Sudoku formulas), and that every Sudoku formula is a finite string of symbols.

Sudoku formulas (as all the propositional formulas)
are often cumbersome. To make conveying ideas using this language
easier we adopt several standard conventions.
\begin{description}
  \item[Convention 1.] We do not write \emph{outer} parentheses in formulas. For example, instead of
     $$
       ((\lnot(r_1c_1 \not\approx 1) \land \lnot(r_1c_2 \not\approx 3)) \Rightarrow (r_1c_3 \not\approx 1 \land r_1c_3 \not\approx 3))
     $$
     we shall write
     $$
       (\lnot(r_1c_1 \not\approx 1) \land \lnot(r_1c_2 \not\approx 3)) \Rightarrow (r_1c_3 \not\approx 1 \land r_1c_3 \not\approx 3).
     $$
  \item[Convention 2.] Logical connectives $\land$ and $\lor$ have higher precedence (they bind more strongly)
    than $\Rightarrow$ and $\Leftrightarrow$, so the above formula can be further simplified to:
    $$
      \lnot(r_1c_1 \not\approx 1) \land \lnot(r_1c_2 \not\approx 3) \Rightarrow (r_1c_3 \not\approx 1 \land r_1c_3 \not\approx 3).
    $$
  \item[Convention 3.] We shall omit parentheses for long stretches of logical connectives $\land$ and, respectively, $\lor$.
    For example, instead of
    $$
      (\alpha_1 \land (\alpha_2 \land (\alpha_3 \land \alpha_4))) \text{ and resp. } ((\beta_1 \lor \beta_2) \lor (\beta_3 \lor \beta_4))
    $$
    we shall simply write
    $$
      \alpha_1 \land \alpha_2 \land \alpha_3 \land \alpha_4 \text{ and resp. } \beta_1 \lor \beta_2 \lor \beta_3 \lor \beta_4.
    $$
\end{description}

Of course, we shall reintroduce the extra parentheses whenever we need to clarify an expression or increase readability.
Let us now introduce two new symbols that will be useful in expressing properties of Sudoku grids.

\begin{DEF}
  Let $r_i c_j \equiv d$ be a shorthand for:
  $$
    r_i c_j \equiv d: \qquad \bigwedge_{s \ne d} r_i c_j \not\approx s.
  $$
  In that case we say that the cell $r_i c_j$ is \emph{uniquely determined}.
\end{DEF}

For example, $r_2 c_3 \equiv 7$ is a shorthand for
$$
  \begin{array}{l}
  r_2 c_3 \not\approx 1 \land r_2 c_3 \not\approx 2 \land r_2 c_3 \not\approx 3 \land r_2 c_3 \not\approx 4 \land \mathstrut\\
  r_2 c_3 \not\approx 5 \land r_2 c_3 \not\approx 6 \land r_2 c_3 \not\approx 8 \land r_2 c_3 \not\approx 9.
  \end{array}
$$
It turns out that the most efficient way to claim that the only possibility for a cell $r_ic_j$ to be $d$ in Sudoku logic
\emph{is to assert that it cannot take any other value!}

\begin{DEF}
  Let $r_ic_j \parallel r_\ell c_m$ be a shorthand for:
  $$
    r_ic_j \parallel r_\ell c_m: \qquad \bigwedge_{1 \le d \le 9} ((r_i c_j \equiv d \Rightarrow r_\ell c_m \not\approx d) \land 
                                        (r_\ell c_m \equiv d \Rightarrow r_i c_j \not\approx d)).
  $$
\end{DEF}

The formula $r_ic_j \parallel r_\ell c_m$ conveys the message that whenever one of the two cells 
is uniquely determined, the other cannot possibly take the same value.

To see all this in action let us write down the three \emph{Sudoku axioms} using Sudoku formulas:
\begin{description}
  \item[A1.] in every row all the digits have to be distinct,
  \item[A2.] in every column all the digits have to be distinct,
  \item[A3.] in every box all the digits have to be distinct
\end{description}
The axiom A1 can be broken down into nine \emph{row conditions}:
$$
  \text{\bf A1:} \quad R_1 \land R_2 \land \ldots \land R_9
$$
where each row condition $R_i$ asserts that in row $i$ all the digits are distinct. This can be written down compactly as
$$
  R_i: \quad \bigwedge_{1 \le j < m \le 9} r_ic_j \parallel r_ic_m.
$$
Analogously, the axiom A2 can be broken down into nine \emph{column conditions}:
$$
  \text{\bf A2:} \quad C_1 \land C_2 \land \ldots \land C_9
$$
where each column condition $C_j$ asserts that in column $j$ all the digits are distinct. This can be written down compactly as
$$
  C_j: \quad \bigwedge_{1 \le i < \ell \le 9} r_ic_j \parallel r_\ell c_j.
$$
The axiom A3 can also be broken down into nine \emph{box conditions}:
$$
  \text{\bf A3:} \quad B_1 \land B_2 \land \ldots \land B_9
$$
where each box condition $B_n$ asserts that in box $n$ all the digits are distinct. This is not easy to write down
compactly. As an example of what we are dealing with,
formulas $R_1$, $C_2$ and $B_3$ are written down completely in Fig.~\ref{sudoku.fig.formulas}.
Finally, let
$$
  \Sax = \{R_1, R_2, \ldots, R_9, C_1, C_2, \ldots, C_9, B_1, B_2, \ldots, B_9\}
$$
be the set of Sudoku formulas that captures the axioms of classic Sudoku.

\begin{figure}
  \centering
  \begin{tabular}{r@{\,}c@{\,}l}
    $R_1$ &=& \phantom{$\land$ }$(r_1 c_1 \parallel r_1 c_2) \land (r_1 c_1 \parallel r_1 c_3) \land (r_1 c_1 \parallel r_1 c_4) \land (r_1 c_1 \parallel r_1 c_5)$\\
          & &   $\mathstrut\land (r_1 c_1 \parallel r_1 c_6) \land (r_1 c_1 \parallel r_1 c_7) \land (r_1 c_1 \parallel r_1 c_8) \land (r_1 c_1 \parallel r_1 c_9)$\\
          & &   $\mathstrut\land (r_1 c_2 \parallel r_1 c_3) \land (r_1 c_2 \parallel r_1 c_4) \land (r_1 c_2 \parallel r_1 c_5) \land (r_1 c_2 \parallel r_1 c_6)$\\
          & &   $\mathstrut\land (r_1 c_2 \parallel r_1 c_7) \land (r_1 c_2 \parallel r_1 c_8) \land (r_1 c_2 \parallel r_1 c_9) \land (r_1 c_3 \parallel r_1 c_4)$\\
          & &   $\mathstrut\land (r_1 c_3 \parallel r_1 c_5) \land (r_1 c_3 \parallel r_1 c_6) \land (r_1 c_3 \parallel r_1 c_7) \land (r_1 c_3 \parallel r_1 c_8)$\\
          & &   $\mathstrut\land (r_1 c_3 \parallel r_1 c_9) \land (r_1 c_4 \parallel r_1 c_5) \land (r_1 c_4 \parallel r_1 c_6) \land (r_1 c_4 \parallel r_1 c_7)$\\
          & &   $\mathstrut\land (r_1 c_4 \parallel r_1 c_8) \land (r_1 c_4 \parallel r_1 c_9) \land (r_1 c_5 \parallel r_1 c_6) \land (r_1 c_5 \parallel r_1 c_7)$\\
          & &   $\mathstrut\land (r_1 c_5 \parallel r_1 c_8) \land (r_1 c_5 \parallel r_1 c_9) \land (r_1 c_6 \parallel r_1 c_7) \land (r_1 c_6 \parallel r_1 c_8)$\\
          & &   $\mathstrut\land (r_1 c_6 \parallel r_1 c_9) \land (r_1 c_7 \parallel r_1 c_8) \land (r_1 c_7 \parallel r_1 c_9) \land (r_1 c_8 \parallel r_1 c_9)$\\[3mm]
    $C_2$ &=& \phantom{$\land$ }$(r_1 c_2 \parallel r_2 c_2) \land (r_1 c_2 \parallel r_3 c_2) \land (r_1 c_2 \parallel r_4 c_2) \land (r_1 c_2 \parallel r_5 c_2)$\\
          & &   $\mathstrut\land (r_1 c_2 \parallel r_6 c_2) \land (r_1 c_2 \parallel r_7 c_2) \land (r_1 c_2 \parallel r_8 c_2) \land (r_1 c_2 \parallel r_9 c_2)$\\
          & &   $\mathstrut\land (r_2 c_2 \parallel r_3 c_2) \land (r_2 c_2 \parallel r_4 c_2) \land (r_2 c_2 \parallel r_5 c_2) \land (r_2 c_2 \parallel r_6 c_2)$\\
          & &   $\mathstrut\land (r_2 c_2 \parallel r_7 c_2) \land (r_2 c_2 \parallel r_8 c_2) \land (r_2 c_2 \parallel r_9 c_2) \land (r_3 c_2 \parallel r_4 c_2)$\\
          & &   $\mathstrut\land (r_3 c_2 \parallel r_5 c_2) \land (r_3 c_2 \parallel r_6 c_2) \land (r_3 c_2 \parallel r_7 c_2) \land (r_3 c_2 \parallel r_8 c_2)$\\
          & &   $\mathstrut\land (r_3 c_2 \parallel r_9 c_2) \land (r_4 c_2 \parallel r_5 c_2) \land (r_4 c_2 \parallel r_6 c_2) \land (r_4 c_2 \parallel r_7 c_2)$\\
          & &   $\mathstrut\land (r_4 c_2 \parallel r_8 c_2) \land (r_4 c_2 \parallel r_9 c_2) \land (r_5 c_2 \parallel r_6 c_2) \land (r_5 c_2 \parallel r_7 c_2)$\\
          & &   $\mathstrut\land (r_5 c_2 \parallel r_8 c_2) \land (r_5 c_2 \parallel r_9 c_2) \land (r_6 c_2 \parallel r_7 c_2) \land (r_6 c_2 \parallel r_8 c_2)$\\
          & &   $\mathstrut\land (r_6 c_2 \parallel r_9 c_2) \land (r_7 c_2 \parallel r_8 c_2) \land (r_7 c_2 \parallel r_9 c_2) \land (r_8 c_2 \parallel r_9 c_2)$\\[3mm]
    $B_3$ &=& \phantom{$\land$ }$(r_1 c_7 \parallel r_1 c_8) \land (r_1 c_7 \parallel r_1 c_9) \land (r_1 c_7 \parallel r_2 c_7) \land (r_1 c_7 \parallel r_2 c_8)$\\
          & &   $\mathstrut\land (r_1 c_7 \parallel r_2 c_9) \land (r_1 c_7 \parallel r_3 c_7) \land (r_1 c_7 \parallel r_3 c_8) \land (r_1 c_7 \parallel r_3 c_9)$\\
          & &   $\mathstrut\land (r_1 c_8 \parallel r_1 c_9) \land (r_1 c_8 \parallel r_2 c_7) \land (r_1 c_8 \parallel r_2 c_8) \land (r_1 c_8 \parallel r_2 c_9)$\\
          & &   $\mathstrut\land (r_1 c_8 \parallel r_3 c_7) \land (r_1 c_8 \parallel r_3 c_8) \land (r_1 c_8 \parallel r_3 c_9) \land (r_1 c_9 \parallel r_2 c_7)$\\
          & &   $\mathstrut\land (r_1 c_9 \parallel r_2 c_8) \land (r_1 c_9 \parallel r_2 c_9) \land (r_1 c_9 \parallel r_3 c_7) \land (r_1 c_9 \parallel r_3 c_8)$\\
          & &   $\mathstrut\land (r_1 c_9 \parallel r_3 c_9) \land (r_2 c_7 \parallel r_2 c_8) \land (r_2 c_7 \parallel r_2 c_9) \land (r_2 c_7 \parallel r_3 c_7)$\\
          & &   $\mathstrut\land (r_2 c_7 \parallel r_3 c_8) \land (r_2 c_7 \parallel r_3 c_9) \land (r_2 c_8 \parallel r_2 c_9) \land (r_2 c_8 \parallel r_3 c_7)$\\
          & &   $\mathstrut\land (r_2 c_8 \parallel r_3 c_8) \land (r_2 c_8 \parallel r_3 c_9) \land (r_2 c_9 \parallel r_3 c_7) \land (r_2 c_9 \parallel r_3 c_8)$\\
          & &   $\mathstrut\land (r_2 c_9 \parallel r_3 c_9) \land (r_3 c_7 \parallel r_3 c_8) \land (r_3 c_7 \parallel r_3 c_9) \land (r_3 c_8 \parallel r_3 c_9)$
  \end{tabular}
  \caption{Formulas $R_1$, $C_2$ and $B_3$}
  \label{sudoku.fig.formulas}
\end{figure}

\section{Sudoku logic: Semantics}
\label{sudoku.sec.Sem}

In any formal deductive system semantics provides a context in which formulas of the system can be interpreted so that
the validity of formulas can be established. This is usually expressed as a relation $\models$ between 
contexts and formulas where
$$
  C \models \alpha
$$
means that the formula $\alpha$ is true in the context $C$. We read this as ``$\alpha$ is true in $C$'',
or ``$\alpha$ holds in $C$'', or ``$C$ is a model of $\alpha$'', or ``$C$ satisfies $\alpha$''.

In case of classic Sudoku contexts will given by (partially filled)
Sudoku grids, while formulas will be Sudoku formulas.

\begin{DEF}
  A \emph{grid} is a $9 \times 9$ matrix
  $$
    \calA = [A_{ij}]_{9\times 9} = \begin{bmatrix}
            A_{11} & A_{12} & A_{13} & A_{14} & A_{15} & A_{16} & A_{17} & A_{18} & A_{19}\\ 
            A_{21} & A_{22} & A_{23} & A_{24} & A_{25} & A_{26} & A_{27} & A_{28} & A_{29}\\
            A_{31} & A_{32} & A_{33} & A_{34} & A_{35} & A_{36} & A_{37} & A_{38} & A_{39}\\
            A_{41} & A_{42} & A_{43} & A_{44} & A_{45} & A_{46} & A_{47} & A_{48} & A_{49}\\
            A_{51} & A_{52} & A_{53} & A_{54} & A_{55} & A_{56} & A_{57} & A_{58} & A_{59}\\
            A_{61} & A_{62} & A_{63} & A_{64} & A_{65} & A_{66} & A_{67} & A_{68} & A_{69}\\
            A_{71} & A_{72} & A_{73} & A_{74} & A_{75} & A_{76} & A_{77} & A_{78} & A_{79}\\
            A_{81} & A_{82} & A_{83} & A_{84} & A_{85} & A_{86} & A_{87} & A_{88} & A_{89}\\
            A_{91} & A_{92} & A_{93} & A_{94} & A_{95} & A_{96} & A_{97} & A_{98} & A_{99}
        \end{bmatrix}
  $$
  where each $A_{ij}$ is a nonempty subset of the set $\{1, 2, \ldots, 9\}$.

  We say that the \emph{cell $r_i c_j$ in $\calA$ is uniquely determined} if $|A_{ij}| = 1$.
  Otherwise it is not uniquely determined.

  If every cell in $\calA$ is uniquely determined then $\calA$ will be referred to as a \emph{full grid}.
\end{DEF}

The idea in this definition is that each $A_{ij}$ represents the set of possible digits the cell $r_i c_j$ can take.
Note that our definition requires that there is at least one possibility for every cell.
Note also that a grid is \emph{not necessarily} a Sudoku grid (because in the definition we do not forbid repetitions of digits
in rows, columns and boxes, for example).

\begin{DEF}\label{sudoku.def.models}
  If $\calA = [A_{ij}]_{9\times 9}$ is a grid and $\phi$ is a Sudoku formula, then we define
  $
    \calA \models \phi
  $
  inductively as follows:
  \begin{itemize}\itemsep -1pt
    \item $\calA \models \top$ is always true; $\calA \models \bot$ is never true;
    \item $\calA \models r_ic_j \not\approx d$ if and only if $d \notin A_{ij}$;
    \item $\calA \models \lnot \alpha$ if and only if it is not the case that $\calA \models \alpha$;
    \item $\calA \models \alpha \land \beta$ if and only if $\calA \models \alpha$ and $\calA \models \beta$;
    \item $\calA \models \alpha \lor \beta$ if and only if $\calA \models \alpha$ or $\calA \models \beta$, or both;
    \item $\calA \models \alpha \Rightarrow \beta$ if and only if $\calA \models \beta$ assuming $\calA \models \alpha$;
    \item $\calA \models \alpha \Leftrightarrow \beta$ if and only if $\calA \models \alpha \Rightarrow \beta$ and $\calA \models \beta \Rightarrow \alpha$.
  \end{itemize}
\end{DEF}

\begin{DEF}\label{sudoku.def.models-set}
  If $\calA$ is a grid and $\Phi$ is a set of Sudoku formulas, then $\calA \models \Phi$ means that
  $\calA \models \phi$ for every $\phi \in \Phi$. We take by definition that $\calA \models \0$ for every grid~$\calA$.
\end{DEF}

The idea behind the fact that every grid models the empty set of Sudoku formulas
is simple: since the empty set of formulas imposes no restrictions on the structure of the grid,
it is satisfied by every grid.

\begin{DEF}
  A grid $\calA$ is a \emph{Sudoku grid} if $\calA \models \Sax$.
\end{DEF}

As a consequence of the definition of $\parallel$ we have that
in a Sudoku grid not all the cells are required to be uniquely determined,
but the unique values that appear in the same row (column and box, respectively) have to be distinct, and
they cannot appear as possible values for cells that are not uniquely determined.
To see why this is so, let $\calA$ be the following Sudoku grid:
\begin{center}
    \def\a{$\substack{2 \; 4 \\ 6 \; 8}$}
    \def\b{$\substack{5 \; 9}$}
    \def\c{$\substack{7 \; 9}$}
    \def\d{$\substack{2 \; 4}$}
    \def\e{$\substack{2 \; 6}$}
    \def\f{$\substack{6 \; 8}$}
    \def\h#1{{\large\textbf{#1}}}
\begin{pgfpicture}
  \pgfsetxvec{\pgfpoint{\acadpgfunit}{0pt}}
  \pgfsetyvec{\pgfpoint{0pt}{\acadpgfunit}}
  \pgfsetlinewidth{\acadpgflinewidth}
  \pgftransformshift{\pgfpointxy{-37.5}{12.5}}

  \begin{pgfscope}
    \pgfsetlinewidth{0.50mm}
    \pgfpathmoveto{\pgfpointxy{150.0}{0.0}}
    \pgfpathlineto{\pgfpointxy{150.0}{250.0}}
    \pgfusepath{stroke}
  \end{pgfscope}
  \begin{pgfscope}
    \pgfsetlinewidth{0.50mm}
    \pgfpathmoveto{\pgfpointxy{150.0}{250.0}}
    \pgfpathlineto{\pgfpointxy{1050.0}{250.0}}
    \pgfusepath{stroke}
  \end{pgfscope}
  \begin{pgfscope}
    \pgfpathmoveto{\pgfpointxy{250.0}{0.0}}
    \pgfpathlineto{\pgfpointxy{250.0}{250.0}}
    \pgfusepath{stroke}
  \end{pgfscope}
  \begin{pgfscope}
    \pgfpathmoveto{\pgfpointxy{350.0}{0.0}}
    \pgfpathlineto{\pgfpointxy{350.0}{250.0}}
    \pgfusepath{stroke}
  \end{pgfscope}
  \begin{pgfscope}
    \pgfsetlinewidth{0.50mm}
    \pgfpathmoveto{\pgfpointxy{450.0}{0.0}}
    \pgfpathlineto{\pgfpointxy{450.0}{250.0}}
    \pgfusepath{stroke}
  \end{pgfscope}
  \begin{pgfscope}
    \pgfpathmoveto{\pgfpointxy{150.0}{150.0}}
    \pgfpathlineto{\pgfpointxy{1050.0}{150.0}}
    \pgfusepath{stroke}
  \end{pgfscope}
  \begin{pgfscope}
    \pgfpathmoveto{\pgfpointxy{550.0}{0.0}}
    \pgfpathlineto{\pgfpointxy{550.0}{250.0}}
    \pgfusepath{stroke}
  \end{pgfscope}
  \begin{pgfscope}
    \pgfpathmoveto{\pgfpointxy{650.0}{0.0}}
    \pgfpathlineto{\pgfpointxy{650.0}{250.0}}
    \pgfusepath{stroke}
  \end{pgfscope}
  \begin{pgfscope}
    \pgfsetlinewidth{0.50mm}
    \pgfpathmoveto{\pgfpointxy{750.0}{0.0}}
    \pgfpathlineto{\pgfpointxy{750.0}{250.0}}
    \pgfusepath{stroke}
  \end{pgfscope}
  \begin{pgfscope}
    \pgfpathmoveto{\pgfpointxy{850.0}{0.0}}
    \pgfpathlineto{\pgfpointxy{850.0}{250.0}}
    \pgfusepath{stroke}
  \end{pgfscope}
  \begin{pgfscope}
    \pgfpathmoveto{\pgfpointxy{950.0}{0.0}}
    \pgfpathlineto{\pgfpointxy{950.0}{250.0}}
    \pgfusepath{stroke}
  \end{pgfscope}
  \begin{pgfscope}
    \pgfsetlinewidth{0.50mm}
    \pgfpathmoveto{\pgfpointxy{1050.0}{0.0}}
    \pgfpathlineto{\pgfpointxy{1050.0}{250.0}}
    \pgfusepath{stroke}
  \end{pgfscope}
  \begin{pgfscope}
    \pgfpathmoveto{\pgfpointxy{150.0}{50.0}}
    \pgfpathlineto{\pgfpointxy{1050.0}{50.0}}
    \pgfusepath{stroke}
  \end{pgfscope}
  \pgftext[at={\pgfpointxy{200.0}{200.0}}]{\h1}
  \pgftext[at={\pgfpointxy{300.0}{200.0}}]{\h3}
  \pgftext[bottom,at={\pgfpointxy{200.0}{262.0}}]{$c_1$}
  \pgftext[bottom,at={\pgfpointxy{300.0}{262.0}}]{$c_2$}
  \pgftext[bottom,at={\pgfpointxy{400.0}{262.0}}]{$c_3$}
  \pgftext[right,at={\pgfpointxy{138.0}{200.0}}]{$r_1$}
  \pgftext[bottom,at={\pgfpointxy{500.0}{262.0}}]{$c_4$}
  \pgftext[bottom,at={\pgfpointxy{600.0}{262.0}}]{$c_5$}
  \pgftext[bottom,at={\pgfpointxy{700.0}{262.0}}]{$c_6$}
  \pgftext[bottom,at={\pgfpointxy{800.0}{262.0}}]{$c_7$}
  \pgftext[bottom,at={\pgfpointxy{900.0}{262.0}}]{$c_8$}
  \pgftext[bottom,at={\pgfpointxy{1000.0}{262.0}}]{$c_9$}
  \pgftext[right,at={\pgfpointxy{138.0}{100.0}}]{$r_2$}
  \pgftext[at={\pgfpointxy{500.0}{100.0}}]{\h7}
  \pgftext[at={\pgfpointxy{900.0}{200.0}}]{\h5}
  \pgftext[at={\pgfpointxy{800.0}{100.0}}]{\h3}
  \pgftext[at={\pgfpointxy{1000.0}{100.0}}]{\h1}
  \pgftext[at={\pgfpointxy{400.0}{200.0}}]{\a}
  \pgftext[at={\pgfpointxy{200.0}{100.0}}]{\a}
  \pgftext[at={\pgfpointxy{300.0}{100.0}}]{\a}
  \pgftext[at={\pgfpointxy{400.0}{100.0}}]{\a}
  \pgftext[at={\pgfpointxy{900.0}{100.0}}]{\a}
  \pgftext[at={\pgfpointxy{600.0}{100.0}}]{\b}
  \pgftext[at={\pgfpointxy{700.0}{100.0}}]{\b}
  \pgftext[at={\pgfpointxy{800.0}{200.0}}]{\c}
  \pgftext[at={\pgfpointxy{1000.0}{200.0}}]{\c}
  \pgftext[at={\pgfpointxy{500.0}{200.0}}]{\d}
  \pgftext[at={\pgfpointxy{600.0}{200.0}}]{\e}
  \pgftext[at={\pgfpointxy{700.0}{200.0}}]{\f}
\end{pgfpicture}
\end{center}
Some cells in this grid are uniquely determined (for example, $A_{12} = \{3\}$)
and others are not (for example, $A_{14} = \{2, 4\}$). Let us show that $A \models R_1$. It is obvious that
$\calA \models r_1c_1 \parallel r_1c_2$ because $|A_{11}| = |A_{12}| = 1$ and the unique values in $A_{11}$ and $A_{12}$
are not equal.
To see that $\calA \models r_1c_1 \parallel r_1c_3$ note that
the cell $r_1c_1$ is uniquely determined and the digit in $r_1c_1$ does not appear among the possible values for $r_1 c_3$.
Finally, $\calA \models r_1c_3 \parallel r_1c_4$ as none of the two cells is uniquely determined.
These three ideas suffice to prove that the remaining 33 Sudoku propositions that are listed in $R_1$ are
also satisfied.

\begin{DEF}
  Let $\Phi$ be a set of Sudoku formulas. Then
  $$
    \models \Phi
  $$
  means that $\calS \models \Phi$ for every full Sudoku grid $\calS$.
\end{DEF}

Let us now specify what a solution to a Sudoku puzzle is.

\begin{DEF}
  Let $\calA = [A_{ij}]_{9 \times 9}$ and $\calB = [B_{ij}]_{9 \times 9}$ be grids.
  We say that \emph{$\calB$ is compatible with $\calA$} if
  $B_{ij} \subseteq A_{ij}$ for all $1 \le i,j \le 9$.
\end{DEF}

\begin{DEF}
  Let $\calA = [A_{ij}]_{9 \times 9}$ and $\calS = [S_{ij}]_{9 \times 9}$ be grids.
  We say that \emph{$\calS$ is a solution for $\calA$} if $\calS$ is a full grid compatible with~$\calA$.
\end{DEF}

\begin{DEF}
  Let $\Phi$ and $\Psi$ be two sets of Sudoku propositions (that is, formulas of the form $r_i c_j \not\approx d$).
  Then
  $$
    \Phi \models \Psi
  $$
  means that for every Sudoku grid $\calA$, if $\calA \models \Phi$ then $\calS \models \Psi$
  for every full Sudoku grid $\calS$ which is a solution of~$\calA$.
  We say that the formulas in $\Psi$ are \emph{semantic consequences of} $\Phi$.
\end{DEF}

\paragraph{Convention.}
Instead of $\Phi \models \{\alpha\}$ we simply write $\Phi \models \alpha$, and instead of
$\{\alpha\} \models \Psi$ we simply write $\alpha \models \Phi$. Also, instead of
$\{\alpha_1, \ldots, \alpha_n\} \models \{\beta_1, \ldots, \beta_m\}$ we simply write
$\alpha_1, \ldots, \alpha_n \models \beta_1, \ldots, \beta_m$.

\bigskip

Note that the above definition overloads the symbol $\models$, and this is standard in logic.
However, there is no danger of confusion: there is a clear distinction between the statement of the form
$$
  \text{grid} \models \text{set of arbitrary Sudoku formulas},
$$
the statement of the form
$$
  \models \text{set of arbitrary Sudoku formulas},
$$
and the statement of the form
$$
  \text{set of Sudoku propositions} \models \text{set of Sudoku propositions}.
$$

\begin{DEF}
  For a Sudoku grid $\calA = [A_{ij}]_{9 \times 9}$ let $\Cns(\calA)$ denote the set of all the
  \emph{Sudoku propositions valid in $\calA$}:
  $$
    \Cns(\calA) = \big\{ r_i c_j \not\approx d : 1 \le i, j, d \le 9 \text{ and } d \notin A_{ij} \}.
  $$
\end{DEF}

In particular, if the cell $r_i c_j$ is uniquely determined and $A_{ij} = \{d\}$ then for all $s \ne d$ we have that
$r_ic_j \not\approx s \in \Cns(\calA)$. If, on the other hand, the cell $r_ic_j$ is not uniquely determined,
the fact that $\calA$ is a Sudoku grid (that is, $\calA \models \Sax$) implies the following:
\begin{itemize}
  \item if there is an $m \ne j$ such that $A_{im} = \{d\}$ then $r_ic_j \not\approx d \in \Cns(\calA)$;
  \item if there is an $\ell \ne i$ such that $A_{\ell j} = \{d\}$ then $r_ic_j \not\approx d \in \Cns(\calA)$;
  \item if there is a pair $(\ell, m) \ne (i, j)$ such that $r_i c_j$ and $r_\ell c_m$ are in the same box
        and $A_{\ell m} = \{d\}$ then $r_ic_j \not\approx d \in \Cns(\calA)$.
\end{itemize}

\begin{LEM}\label{sudoku.lem.SmodelsCnsA}
  Let $\calA$ and $\calB$ be Sudoku grids and assume that $\calB$ is compatible with $\calA$.
  Then:
  
  $(a)$ $\Cns(\calA) \subseteq \Cns(\calB)$.

  $(b)$ $\calB \models \Cns(\calA)$.

  $(c)$ If $\calA \models r_ic_j \equiv d$ then $\calB \models r_ic_j \equiv d$.
\end{LEM}
\begin{proof}
  $(a)$ Let $\calA = [A_{ij}]_{9 \times 9}$ and $\calB = [S_{ij}]_{9 \times 9}$.
  Take any formula $r_i c_j \not\approx d \in \Cns(\calA)$. Then $d \notin A_{ij}$ by the definition of $\Cns(\calA)$.
  Since $\calB$ is compatible with $\calA$, we also know that $S_{ij} \subseteq A_{ij}$. Therefore,
  $d \notin S_{ij}$ and thus $r_i c_j \not\approx d  \in \Cns(\calB)$.

  $(b)$ and $(c)$ follow immediately from $(a)$.
\end{proof}

As a demonstration of the theory developed so far, we propose a series of exercises which validate
the X-wing rule. Let $\XRow(i, j, m, d)$ denote the following set of Sudoku propositions:
$$
  \XRow(i, j, m, d) = \{ r_i c_s \not\approx d : s \ne j, m \}
$$
which claims that in row $i$ the digit $d$ can occur in none but two cells, $r_i c_j$ and $r_i c_m$.
Analogously, let $\XCol(j, i, \ell, d)$ denote the following set of Sudoku propositions:
$$
  \XCol(j, i, \ell, d) = \{ r_t c_j \not\approx d : t \ne i, \ell \}
$$
which encodes an analogous claim for columns.
Then the X-wing rule for rows is the following claim:

\begin{quote}
  \textbf{X-wing rule for rows.}
  Fix $1 \le i, j, \ell, m, d \le 9$ so that $i \ne \ell$ and $j \ne m$. Then
  $$
    \XRow(i, j, m, d) \cup \XRow(\ell, j, m, d) \models \XCol(j, i, \ell, d) \cup \XCol(m, i, \ell, d).
  $$
\end{quote}

\begin{ZAD}
  Prove the X-wing rule for rows. In other words, prove the following: if $\calA$ is a Sudoku grid such that
  $\calA \models \XRow(i, j, m, d) \cup \XRow(\ell, j, m, d)$ and if $\calS$ is a Sudoku grid which is a
  solution of $\calA$ then $\calS \models \XCol(j, i, \ell, d) \cup \XCol(m, i, \ell, d)$.
\end{ZAD}

\begin{ZAD}
  Formulate and prove the X-wing rule for columns.
\end{ZAD}

\begin{ZAD}
  Formulate and prove the jellyfish rule for rows and columns.
\end{ZAD}

\section{Deducibility}
\label{sudoku.sec.Deduc}

A deduction is a highly formalized procedure of establishing that a particular piece of information
must be a logical consequence of certain assumptions. In every deduction we start from universally accepted truths
called \emph{axioms}, and then starting from a set of \emph{premisses} we then produce a sequence
of intermediate statements that lead to a \emph{conclusion}. Along the way, we obtain new statements from those
that appear prior to it in the deduction using some prescribed \emph{rules of deduction}. 

Usually there are two kinds of axioms: logical axioms and domain-specific axioms.
\emph{Logical axioms} are universally true across mathematics. They govern our reasoning and
ensure that in every part of mathematics we use the same rigorous procedures, so that conclusions we make in euclidean geometry
are as reliable as conclusions we make in number theory.

Using logical axioms only nothing can be said about euclidean geometry, say. What is a straight line? What is a point?
In what ways can a straight line and a point interact? None of this is encoded in logical axioms.
The choice of \emph{domain-specific axioms} tells us what is the mathematical theory we are working in.
Their principal task is to introduce the fundamental objects of the theory (such as points and lines) and
their defining properties (such as: for every pair of distinct points there is precisely one line that passes through both of them).

Finally, problem-specific premisses are the assumptions about a particular configuration for which
we are trying to infer new properties.

For example, in the proof of the Pythagorean Theorem:
\begin{quote}
  \textbf{Pythagorean Theorem.} \textsl{In every right-angled triangle, the square of the length of one side of the triangle
  equals the sum of squares of lengths of the other two sides.}
\end{quote}
\noindent
the logical axioms are the axioms of mathematical logic (these are usually taught at schools
and universities as ``proof strategies'' or ``proof techniques'').
The domain-specific axioms are the axioms of euclidean geometry. This is not specified explicitly in the
statement of the theorem, but anyone can safely infer that from the historical context. Finally,
problem-specific assumptions are that we are given a triangle in which one of the angles is right.

The \emph{logical axioms} of Sudoku logic are all Sudoku formulas that are derived from tautologies.
A \emph{tautology} is a universally true propositional formula, such as
$$
  (\lnot q \Rightarrow \lnot p) \Rightarrow (p \Rightarrow q) \quad \text{and}\quad p \Rightarrow(q \Rightarrow p \land q).
$$
A \emph{Sudoku formula derived from a tautology} is any formula that we can obtain from a tautology by replacing
letters with arbitrary Sudoku propositions. For example, the following two Sudoku axioms are examples of
logical Sudoku axioms derived from the above two tautologies:
\begin{gather*}
  (\lnot(r_1 c_2 \not\approx 4) \Rightarrow \lnot(r_1 c_3 \not\approx 7)) \Rightarrow (r_1 c_3 \not\approx 7 \Rightarrow r_1 c_2 \not\approx 4) \\
  \text{and}\quad r_1 c_2 \not\approx 4 \Rightarrow(r_1 c_2 \not\approx 5 \Rightarrow r_1 c_2 \not\approx 4 \land r_1 c_2 \not\approx 5).
\end{gather*}

A \emph{Sudoku axiom} is a Sudoku formula which holds in every full Sudoku grid.
Since tautologies are universally true, every Sudoku formula derived from a tautology is also a Sudoku axiom.

Let $\Dax$ denote the set of all Sudoku axioms.
The set $\Dax$ contains $\Sax$ and much more.
For example, each of the formulas
$$
  \qquad \lnot(r_i c_j \not\approx d) \Leftrightarrow r_i c_j \equiv d,
$$
where $1 \le i, j, d \le 9$, is also a Sudoku axiom since it is true in every full Sudoku grid.
Note that there are Sudoku grids where such a formula fails, but they all hold in every \emph{full} Sudoku grid.

In Sudoku logic we use only one \emph{deduction rule -- modus ponens}:
$$
  \DedRuleII{\alpha}{\alpha \Rightarrow \beta}{\beta}{1cm}
$$
This depiction means that if we can logically deduce $\alpha$ and if we can logically deduce $\alpha \Rightarrow \beta$, then
we have the license to logically deduce~$\beta$.

\begin{DEF}
  Let $\Phi$ be a set of Sudoku formulas and let $\alpha$ be a Sudoku formula. We say that \emph{$\alpha$ is logically deducible from $\Phi$}
  and write $\Phi \vdash \alpha$ if we can write down a sequence
  \begin{equation}\label{sudoku.eq.proof}
    \phi_1, \; \phi_2, \; \ldots, \; \phi_n
  \end{equation}
  of formulas such that $\phi_n = \alpha$ and for every $1 \le i \le n$:
  \begin{itemize}\itemsep -1pt
    \item $\phi_i$ is a Sudoku axiom (that is, $\phi_i \in \Dax$), or
    \item $\phi_i$ is a \emph{premise} (or \emph{assumption}; that is, $\phi_i \in \Phi$), or
    \item there exist $1 \le p, q < i$ such that $\phi_i$ can be deduced from $\phi_p$ and $\phi_q$ using modus ponens
          (in other words, $\phi_q$ is $\phi_p \Rightarrow \phi_i$).
  \end{itemize}
  In that case $\Phi$ is the set of \emph{premisses}, $\alpha$ is the \emph{consequence} or \emph{conclusion}, and
  the sequence \eqref{sudoku.eq.proof} is the \emph{proof of $\alpha$ from the premisses in $\Phi$}.

  If $\Phi = \0$ is the empty set of premisses, instead of $\0 \vdash \alpha$ we shall simply write $\vdash \alpha$.
  This means that $\alpha$ is a \emph{Sudoku theorem} because it follows only from the Sudoku axioms. As usual,
  we write $\vdash \Phi$ if $\vdash \alpha$ for all $\alpha \in \Phi$.
\end{DEF}

A deduction system (axioms together with deduction rules) makes sense if it is \emph{sound}.
This means that everything we logically deduce within the system must be true ``in reality''.
We shall now prove that the deduction system of Sudoku logic is sound.

\begin{LEM}\label{sudoku.lem.s-phi-alpha}
    Let $\calS$ be a full Sudoku grid, let $\Phi$ be a set of Sudoku formulas and let $\alpha$ be a Sudoku formula.
    If $\calS \models \Phi$ and $\Phi \vdash \alpha$ then $\calS \models \alpha$.
\end{LEM}
\begin{proof}
    Assume that $\calS \models \Phi$ and $\Phi \vdash \alpha$. Then there is a proof
    $$
        \phi_1, \; \phi_2, \; \ldots, \; \phi_n
    $$
    of $\alpha$. We shall prove that $\calS \models \alpha$ using induction on $n$ -- the length of the proof.

    If $n = 1$ then $\alpha$ is a Sudoku axiom or $\alpha \in \Phi$. In any case, $\calS \models \alpha$.

    Assume that the statement is true for every formula $\phi$ whose proof is shorter than~$n$.
    If $\alpha$ is an axiom or $\alpha \in \Phi$ then obviously $\calS \models \alpha$.
    Assume, therefore, that there exist $1 \le p, q < n$ such that $\phi_i$ can be deduced from $\phi_p$ and $\phi_q$ using modus ponens.
    Then, without loss of generality, we may also assume that $\phi_q$ is $\phi_p \Rightarrow \phi_n$.
    Since $p, q < n$ the induction hypothesis yields that $\calS \models \phi_p$ and $\calS \models \phi_p \Rightarrow \phi_n$.
    By Definition~\ref{sudoku.def.models} the statement $\calS \models \phi_p \Rightarrow \phi_n$ means that we can conclude
    $\calS \models \phi_n$ provided we know that $\calS \models \phi_p$ holds, which we do! Therefore, $\calS \models \phi_n$,
    that is, $\calS \models \alpha$.
  \end{proof}

\begin{THM}[Soundness]\label{sudoku.thm.soundness}
  Let $\Phi$ and $\Psi$ be two sets of Sudoku propositions. If $\Phi \vdash \Psi$ then
  $\Phi \models \Psi$.
\end{THM}
\begin{proof}
    Let $\calA$ be a Sudoku grid such that $\calA \models \Phi$ and let $\calS$ be a solution of~$\calA$.
    Take any $\alpha \in \Psi$.
    Then $\calS \models \Cns(\calA)$ by Lemma~\ref{sudoku.lem.SmodelsCnsA}~$(b)$. Since $\calA \models \Phi$ is another
    way of saying that $\Phi \subseteq \Cns(\calA)$ it follows that $\calS \models \Phi$.
    Now, $\calS \models \Phi$ and $\Phi \vdash \alpha$ gives us that $\calS \models \alpha$ by the lemma above.
    This holds for every $\alpha \in \Psi$, so $\calS \models \Psi$.
\end{proof}

\emph{Completeness} of a deductive system is a complementary requirement which ensures that everything
which is ``semantically justifiable'' can also be logically deduced. More precisely:

\begin{THM}[Completeness]\label{sudoku.thm.completeness}
  Let $\Phi$ and $\Psi$ be two sets of Sudoku propositions. If $\Phi \models \Psi$ then
  $\Phi \vdash \Psi$.
\end{THM}
\begin{proof}
  See Appendix~\ref{sudoku.sec.app-A}.
\end{proof}

\begin{COR}
  Let $\Phi$ and $\Psi$ be two sets of Sudoku propositions. Then $\Phi \vdash \Psi$ if and only if
  $\Phi \models \Psi$.
\end{COR}

Our set of Sudoku axioms is very large. For purely aesthetic reasons we are interested in a smaller set
of axioms from which every Sudoku axiom cam be deduced. We pose this as an open problem.

\begin{PROB}\label{sudoku.prob.completeness}
  Find a reasonable set of Sudoku axioms $\calA x \subseteq \Dax$ such that every other Sudoku axiom can be
  logically deduced from $\calA x$.
\end{PROB}

Now that we have set up the logical system and precisely defined the deduction process in Sudoku logic,
we can rigorously define logically deducible solutions of Sudoku grids and prove the first main result of the paper:
if a grid has a logically deducible solution, then this must be the unique solution of the grid;
conversely, if a grid has a unique solution, then this solution must be logically deducible.
So, uniqueness and deducibility are the same thing!

\begin{DEF}
  Let $\calA$ and $\calB$ be Sudoku grids. We say that $\calB$ is \emph{logically deducible} from $\calA$
  if $\Cns(\calA) \vdash \Cns(\calB)$.

  In particular, if $\calS$ is a solution of $\calA$ such that $\Cns(\calA) \vdash \Cns(\calS)$
  we say that $\calS$ is a \emph{logically deducible solution of $\calA$}. 
\end{DEF}

Let us now show that if a Sudoku grid has a logically deducible solution, then this is the unique solution.

\begin{THM}[Deducibility implies uniqueness]\label{sudoku.thm.deduc=>uniq}
  Let $\calA$ and $\calS$ be Sudoku grids.
  If $\calS$ is a logically deducible solution of $\calA$ then it is the only solution of $\calA$.
\end{THM}
\begin{proof}
  Let $\calS$ be a Sudoku grid which is a logically deducible solution of $\calA$ and assume that 
  there is a Sudoku grid $\calS' \ne \calS$ which is another solution of $\calA$.
  Because $\calS' \ne \calS$ we can find $i$ and $j$ such that $S_{ij} \ne S'_{ij}$.
  Let $S_{ij} = \{d\}$ and $S'_{ij} = \{e\}$ where $d \ne e$.
  Then:
  \begin{equation}\label{sudoku.eq.ricjnotapproxe}
    r_i c_j \not\approx e \in \Cns(\calS).
  \end{equation}
  Since $\calS'$ is a solution of $\calA$ we have that
  $$
    \calS' \models \Cns(\calA)
  $$
  by Lemma~\ref{sudoku.lem.SmodelsCnsA}~$(b)$. On the other hand, $\calS$ is a logically deducible solution of $\calA$
  so, by definition,
  $$
    \Cns(\calA) \vdash \Cns(\calS).
  $$
  Since our deductive system is sound (Theorem~\ref{sudoku.thm.soundness}), it follows that
  $$
    \calS' \models \Cns(\calS).
  $$
  This and \eqref{sudoku.eq.ricjnotapproxe} imply
  $$
    \calS' \models r_i c_j \not\approx e
  $$
  which means that $e \notin S'_{ij}$. Contradiction.
\end{proof}

\begin{COR}
  If a Sudoku grid $\calA$ has several distinct solutions, none of them is logically deducible.
\end{COR}

The completeness of the deductive system makes it possible for us to prove the converse of Theorem~\ref{sudoku.thm.deduc=>uniq}:
if a Sudoku grid has a unique solution, then this solution must be logically deducible.
We start with a lemma.

\begin{LEM}\label{sudoku.lem.BLA}
  Let $\calA$ and $\calB$ be Sudoku grids. If $\calB \models \Cns(\calA)$ and if $\calS$ is a solution of $\calB$
  then $\calS$ is a solution of $\calA$.
\end{LEM}
\begin{proof}
  Let $\calA = [A_{ij}]_{9 \times 9}$, $\calB = [B_{ij}]_{9 \times 9}$ and $\calS = [S_{ij}]_{9 \times 9}$
  be Sudoku grids such that $\calB \models \Cns(\calA)$ and $\calS$ is a solution of $\calB$.
  So, $\calS$ is a full grid and $\calS$ is compatible with $\calB$. To complete the proof it suffices to show that
  $\calB$ is compatible with $\calA$. To this end, assume that $d \notin A_{ij}$ for some $1 \le i, j, d \le 9$.
  Then $r_ic_j \not\approx d \in \Cns(\calA)$, so the assumption that $\calB \models \Cns(\calA)$ implies that
  $\calB \models r_ic_j \not\approx d$. Therefore, $d \notin B_{ij}$. We have thus shown that $d \notin A_{ij} \Rightarrow d \notin B_{ij}$,
  whence $B_{ij} \subseteq A_{ij}$. This completes the proof.
\end{proof}

\begin{THM}[Uniqueness implies deducibility]
  If a Sudoku grid $\calA$ has a unique solution, then this is a logically deducible solution of $\calA$.
\end{THM}
\begin{proof}
  Let $\calS$ be a Sudoku grid which is the unique solution of $\calA$. Let us first show that $\Cns(\calA) \models \Cns(\calS)$.
  Let $\calB$ be any Sudoku grid and let $\calB \models \Cns(\calA)$. Also, let $\calL$ be a solution of $\calB$.
  Then by Lemma~\ref{sudoku.lem.BLA} we get that $\calL$ is a solution of $\calA$. But $\calS$ is the unique
  solution of $\calA$, so it must be the case that $\calL = \calS$. Therefore, $\calL \models \Cns(\calS)$ trivially.
  This completes the proof that $\Cns(\calA) \models \Cns(\calS)$. Theorem~\ref{sudoku.thm.completeness} then
  ensures that $\Cns(\calA) \vdash \Cns(\calS)$, and thus $\calS$ is a logically deducible solution of~$\calA$.
\end{proof}

\begin{COR}[Uniqueness $=$ deducibility]\label{sudoku.cor.quniq=deduc}
  A Sudoku grid has a unique solution if and only if it has a logically deducible solution.
\end{COR}

\section{What is a symmetry?}
\label{sudoku.sec.transf}

It is an intrinsic quality of humanity to be able to observe and admire symmetry.
Highly symmetric patterns have always attracted attention not only of artists, but also of
architects and scientists. For example, the abstract understanding of symmetries in
physics started with Einstein who observed in the early 1900's that the symmetries of physical
laws are the primary features of nature. Invariance under abstract symmetries has played
a crucial role in physics ever since.

Mathematics as an abstract theory of structures has a way of understanding and analyzing symmetry which,
as many great ideas, stems from geometry. From the point of view of Euclidean geometry
a plane figure is symmetric if there are nontrivial rigid motions that map the figure onto itself.
For example, the circle is the most symmetric of all (it has infinitely many symmetries),
while a square has only 8 symmetries -- 4 rotations about the center ($0^\circ$, $90^\circ$,
$180^\circ$ and $270^\circ$) and 4 reflections about the horizontal axis, the vertical axis and the two
diagonals:
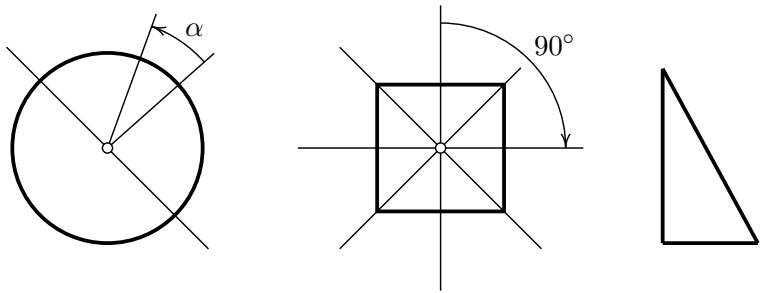
\begin{center}
\begin{pgfpicture}
  \pgfsetxvec{\pgfpoint{\acadpgfunit}{0pt}}
  \pgfsetyvec{\pgfpoint{0pt}{\acadpgfunit}}
  \pgfsetlinewidth{\acadpgflinewidth}
  \pgftransformshift{\pgfpointxy{52.5892}{-202.834}}

  \begin{pgfscope}
    \pgfsetlinewidth{0.50mm}
    \pgfpathellipse{\pgfpointxy{175.0}{450.0}}{\pgfpointxy{150.0}{0.0}}{\pgfpointxy{0.0}{150.0}}
    \pgfusepath{stroke}
  \end{pgfscope}
  \begin{pgfscope}
    \pgfsetlinewidth{0.50mm}
    \pgfpathmoveto{\pgfpointxy{1050.0}{575.0}}
    \pgfpathlineto{\pgfpointxy{1050.0}{300.0}}
    \pgfusepath{stroke}
  \end{pgfscope}
  \begin{pgfscope}
    \pgfsetlinewidth{0.50mm}
    \pgfpathmoveto{\pgfpointxy{1050.0}{300.0}}
    \pgfpathlineto{\pgfpointxy{1200.0}{300.0}}
    \pgfusepath{stroke}
  \end{pgfscope}
  \begin{pgfscope}
    \pgfsetlinewidth{0.50mm}
    \pgfpathmoveto{\pgfpointxy{1200.0}{300.0}}
    \pgfpathlineto{\pgfpointxy{1050.0}{575.0}}
    \pgfusepath{stroke}
  \end{pgfscope}
  \begin{pgfscope}
    \pgfpathmoveto{\pgfpointxy{15.901}{609.099}}
    \pgfpathlineto{\pgfpointxy{334.099}{290.901}}
    \pgfusepath{stroke}
  \end{pgfscope}
  \begin{pgfscope}
    \pgfpathmoveto{\pgfpointxy{175.0}{450.0}}
    \pgfpathlineto{\pgfpointxy{343.167}{599.482}}
    \pgfusepath{stroke}
  \end{pgfscope}
  \begin{pgfscope}
    \pgfpathmoveto{\pgfpointxy{175.0}{450.0}}
    \pgfpathlineto{\pgfpointxy{251.892}{661.454}}
    \pgfusepath{stroke}
  \end{pgfscope}
  \begin{pgfscope}
    \pgfpathmoveto{\pgfpointxy{327.534}{585.586}}
    \pgfpatharcaxes{41.6335}{70.0169}{\pgfpointxy{204.084}{0.0}}{\pgfpointxy{0.0}{204.084}}
    \pgfusepath{stroke}
  \end{pgfscope}
  \begin{pgfscope}
    \pgfpathmoveto{\pgfpointxy{263.179}{627.709}}
    \pgfpatharcaxes{35.2076}{70.0169}{\pgfpointxy{38.7827}{0.0}}{\pgfpointxy{0.0}{38.7827}}
    \pgfusepath{stroke}
  \end{pgfscope}
  \begin{pgfscope}
    \pgfpathmoveto{\pgfpointxy{244.744}{641.797}}
    \pgfpatharcaxes{250.017}{284.826}{\pgfpointxy{38.7827}{0.0}}{\pgfpointxy{0.0}{38.7827}}
    \pgfusepath{stroke}
  \end{pgfscope}
  \begin{pgfscope}
    \pgfpathmoveto{\pgfpointxy{700.0}{225.0}}
    \pgfpathlineto{\pgfpointxy{700.0}{675.0}}
    \pgfusepath{stroke}
  \end{pgfscope}
  \begin{pgfscope}
    \pgfpathmoveto{\pgfpointxy{475.0}{450.0}}
    \pgfpathlineto{\pgfpointxy{925.0}{450.0}}
    \pgfusepath{stroke}
  \end{pgfscope}
  \begin{pgfscope}
    \pgfsetlinewidth{0.50mm}
    \pgfpathmoveto{\pgfpointxy{600.0}{550.0}}
    \pgfpathlineto{\pgfpointxy{800.0}{550.0}}
    \pgfpathlineto{\pgfpointxy{800.0}{350.0}}
    \pgfpathlineto{\pgfpointxy{600.0}{350.0}}
    \pgfpathclose
    \pgfusepath{stroke}
  \end{pgfscope}
  \begin{pgfscope}
    \pgfpathmoveto{\pgfpointxy{541.028}{608.972}}
    \pgfpathlineto{\pgfpointxy{858.972}{291.028}}
    \pgfusepath{stroke}
  \end{pgfscope}
  \begin{pgfscope}
    \pgfpathmoveto{\pgfpointxy{826.531}{576.531}}
    \pgfpathlineto{\pgfpointxy{541.028}{291.028}}
    \pgfusepath{stroke}
  \end{pgfscope}
  \begin{pgfscope}
    \pgfpathmoveto{\pgfpointxy{897.144}{450.0}}
    \pgfpatharcaxes{0.0}{90.0}{\pgfpointxy{197.144}{0.0}}{\pgfpointxy{0.0}{197.144}}
    \pgfusepath{stroke}
  \end{pgfscope}
  \begin{pgfscope}
    \pgfpathmoveto{\pgfpointxy{904.12}{472.124}}
    \pgfpatharcaxes{145.0}{180.0}{\pgfpointxy{38.571}{0.0}}{\pgfpointxy{0.0}{38.571}}
    \pgfusepath{stroke}
  \end{pgfscope}
  \begin{pgfscope}
    \pgfpathmoveto{\pgfpointxy{897.144}{450.0}}
    \pgfpatharcaxes{0.0}{35.0004}{\pgfpointxy{38.571}{0.0}}{\pgfpointxy{0.0}{38.571}}
    \pgfusepath{stroke}
  \end{pgfscope}
  \begin{pgfscope}
    \pgfsetfillcolor{white}
    \pgfpathellipse{\pgfpointxy{175.0}{450.0}}{\pgfpointxy{8.0}{0.0}}{\pgfpointxy{0.0}{8.0}}
    \pgfusepath{fill,stroke}
  \end{pgfscope}
  \begin{pgfscope}
    \pgfsetfillcolor{white}
    \pgfpathellipse{\pgfpointxy{700.0}{450.0}}{\pgfpointxy{8.0}{0.0}}{\pgfpointxy{0.0}{8.0}}
    \pgfusepath{fill,stroke}
  \end{pgfscope}
  \pgftext[bottom,left,at={\pgfpointxy{297.638}{626.844}}]{$\alpha$}
  \pgftext[bottom,left,at={\pgfpointxy{847.402}{597.402}}]{$90^\circ$}
\end{pgfpicture}
\end{center}
The triangle in the figure above is the least symmetric -- the only rigid motion that maps it
onto itself is the \emph{identity} (which, actually, moves nothing -- it is a \emph{trivial symmetry}).

This geometric intuition easily transposes to other, less geometric situations such as Sudoku grids.
Let us start thinking about the geometry of Sudoku grids 
by identifying \emph{Sudoku transformations}, that is, transformations
that turn a Sudoku grid into another Sudoku grid.

After reflecting a Sudoku grid about one of the four possible axes (the horizontal axis, the vertical axis and the two
diagonals) or rotating it clockwise about its central cell by one of the four possible angles ($0^\circ$, $90^\circ$,
$180^\circ$ or $270^\circ$) we again end up with a Sudoku grid. The distribution of givens in the new Sudoku
grid may differ from the distribution in the original grid, but what we get is again a Sudoku grid, Fig.~\ref{sudoku.fig.rot}.

Note that the $0^\circ$-rotation is the identity (nothing has changed), and that clockwise $90^\circ$-rotation
is the same as counterclockwise $270^\circ$-ro\-ta\-ti\-on. This is why there is no need to discuss counterclockwise
rotations.

The four rotations and the four reflections of a Sudoku grid will be referred to as the \emph{geometric transformations}.

\begin{figure}
  \centering
\begin{pgfpicture}
  \pgfsetxvec{\pgfpoint{\acadpgfunit}{0pt}}
  \pgfsetyvec{\pgfpoint{0pt}{\acadpgfunit}}
  \pgfsetlinewidth{\acadpgflinewidth}
  \pgftransformshift{\pgfpointxy{28.0363}{43.4741}}

  \begin{pgfscope}
    \pgfsetlinewidth{0.50mm}
    \pgfpathmoveto{\pgfpointxy{0.0}{0.0}}
    \pgfpathlineto{\pgfpointxy{0.0}{450.0}}
    \pgfusepath{stroke}
  \end{pgfscope}
  \begin{pgfscope}
    \pgfpathmoveto{\pgfpointxy{50.0}{0.0}}
    \pgfpathlineto{\pgfpointxy{50.0}{450.0}}
    \pgfusepath{stroke}
  \end{pgfscope}
  \begin{pgfscope}
    \pgfpathmoveto{\pgfpointxy{0.0}{50.0}}
    \pgfpathlineto{\pgfpointxy{450.0}{50.0}}
    \pgfusepath{stroke}
  \end{pgfscope}
  \begin{pgfscope}
    \pgfsetlinewidth{0.50mm}
    \pgfpathmoveto{\pgfpointxy{0.0}{0.0}}
    \pgfpathlineto{\pgfpointxy{450.0}{0.0}}
    \pgfusepath{stroke}
  \end{pgfscope}
  \begin{pgfscope}
    \pgfpathmoveto{\pgfpointxy{100.0}{0.0}}
    \pgfpathlineto{\pgfpointxy{100.0}{450.0}}
    \pgfusepath{stroke}
  \end{pgfscope}
  \begin{pgfscope}
    \pgfsetlinewidth{0.50mm}
    \pgfpathmoveto{\pgfpointxy{150.0}{0.0}}
    \pgfpathlineto{\pgfpointxy{150.0}{450.0}}
    \pgfusepath{stroke}
  \end{pgfscope}
  \begin{pgfscope}
    \pgfpathmoveto{\pgfpointxy{200.0}{0.0}}
    \pgfpathlineto{\pgfpointxy{200.0}{450.0}}
    \pgfusepath{stroke}
  \end{pgfscope}
  \begin{pgfscope}
    \pgfpathmoveto{\pgfpointxy{250.0}{0.0}}
    \pgfpathlineto{\pgfpointxy{250.0}{450.0}}
    \pgfusepath{stroke}
  \end{pgfscope}
  \begin{pgfscope}
    \pgfsetlinewidth{0.50mm}
    \pgfpathmoveto{\pgfpointxy{300.0}{0.0}}
    \pgfpathlineto{\pgfpointxy{300.0}{450.0}}
    \pgfusepath{stroke}
  \end{pgfscope}
  \begin{pgfscope}
    \pgfpathmoveto{\pgfpointxy{350.0}{0.0}}
    \pgfpathlineto{\pgfpointxy{350.0}{450.0}}
    \pgfusepath{stroke}
  \end{pgfscope}
  \begin{pgfscope}
    \pgfpathmoveto{\pgfpointxy{400.0}{0.0}}
    \pgfpathlineto{\pgfpointxy{400.0}{450.0}}
    \pgfusepath{stroke}
  \end{pgfscope}
  \begin{pgfscope}
    \pgfsetlinewidth{0.50mm}
    \pgfpathmoveto{\pgfpointxy{450.0}{0.0}}
    \pgfpathlineto{\pgfpointxy{450.0}{450.0}}
    \pgfusepath{stroke}
  \end{pgfscope}
  \begin{pgfscope}
    \pgfpathmoveto{\pgfpointxy{0.0}{100.0}}
    \pgfpathlineto{\pgfpointxy{450.0}{100.0}}
    \pgfusepath{stroke}
  \end{pgfscope}
  \begin{pgfscope}
    \pgfsetlinewidth{0.50mm}
    \pgfpathmoveto{\pgfpointxy{0.0}{150.0}}
    \pgfpathlineto{\pgfpointxy{450.0}{150.0}}
    \pgfusepath{stroke}
  \end{pgfscope}
  \begin{pgfscope}
    \pgfpathmoveto{\pgfpointxy{0.0}{200.0}}
    \pgfpathlineto{\pgfpointxy{450.0}{200.0}}
    \pgfusepath{stroke}
  \end{pgfscope}
  \begin{pgfscope}
    \pgfpathmoveto{\pgfpointxy{0.0}{250.0}}
    \pgfpathlineto{\pgfpointxy{450.0}{250.0}}
    \pgfusepath{stroke}
  \end{pgfscope}
  \begin{pgfscope}
    \pgfsetlinewidth{0.50mm}
    \pgfpathmoveto{\pgfpointxy{0.0}{300.0}}
    \pgfpathlineto{\pgfpointxy{450.0}{300.0}}
    \pgfusepath{stroke}
  \end{pgfscope}
  \begin{pgfscope}
    \pgfpathmoveto{\pgfpointxy{0.0}{350.0}}
    \pgfpathlineto{\pgfpointxy{450.0}{350.0}}
    \pgfusepath{stroke}
  \end{pgfscope}
  \begin{pgfscope}
    \pgfpathmoveto{\pgfpointxy{0.0}{400.0}}
    \pgfpathlineto{\pgfpointxy{450.0}{400.0}}
    \pgfusepath{stroke}
  \end{pgfscope}
  \begin{pgfscope}
    \pgfsetlinewidth{0.50mm}
    \pgfpathmoveto{\pgfpointxy{0.0}{450.0}}
    \pgfpathlineto{\pgfpointxy{450.0}{450.0}}
    \pgfusepath{stroke}
  \end{pgfscope}
  \begin{pgfscope}
    \pgfsetlinewidth{0.50mm}
    \pgfpathmoveto{\pgfpointxy{750.0}{0.0}}
    \pgfpathlineto{\pgfpointxy{750.0}{450.0}}
    \pgfusepath{stroke}
  \end{pgfscope}
  \begin{pgfscope}
    \pgfpathmoveto{\pgfpointxy{800.0}{0.0}}
    \pgfpathlineto{\pgfpointxy{800.0}{450.0}}
    \pgfusepath{stroke}
  \end{pgfscope}
  \begin{pgfscope}
    \pgfpathmoveto{\pgfpointxy{750.0}{50.0}}
    \pgfpathlineto{\pgfpointxy{1200.0}{50.0}}
    \pgfusepath{stroke}
  \end{pgfscope}
  \begin{pgfscope}
    \pgfsetlinewidth{0.50mm}
    \pgfpathmoveto{\pgfpointxy{750.0}{0.0}}
    \pgfpathlineto{\pgfpointxy{1200.0}{0.0}}
    \pgfusepath{stroke}
  \end{pgfscope}
  \begin{pgfscope}
    \pgfpathmoveto{\pgfpointxy{850.0}{0.0}}
    \pgfpathlineto{\pgfpointxy{850.0}{450.0}}
    \pgfusepath{stroke}
  \end{pgfscope}
  \begin{pgfscope}
    \pgfsetlinewidth{0.50mm}
    \pgfpathmoveto{\pgfpointxy{900.0}{0.0}}
    \pgfpathlineto{\pgfpointxy{900.0}{450.0}}
    \pgfusepath{stroke}
  \end{pgfscope}
  \begin{pgfscope}
    \pgfpathmoveto{\pgfpointxy{950.0}{0.0}}
    \pgfpathlineto{\pgfpointxy{950.0}{450.0}}
    \pgfusepath{stroke}
  \end{pgfscope}
  \begin{pgfscope}
    \pgfpathmoveto{\pgfpointxy{1000.0}{0.0}}
    \pgfpathlineto{\pgfpointxy{1000.0}{450.0}}
    \pgfusepath{stroke}
  \end{pgfscope}
  \begin{pgfscope}
    \pgfsetlinewidth{0.50mm}
    \pgfpathmoveto{\pgfpointxy{1050.0}{0.0}}
    \pgfpathlineto{\pgfpointxy{1050.0}{450.0}}
    \pgfusepath{stroke}
  \end{pgfscope}
  \begin{pgfscope}
    \pgfpathmoveto{\pgfpointxy{1100.0}{0.0}}
    \pgfpathlineto{\pgfpointxy{1100.0}{450.0}}
    \pgfusepath{stroke}
  \end{pgfscope}
  \begin{pgfscope}
    \pgfpathmoveto{\pgfpointxy{1150.0}{0.0}}
    \pgfpathlineto{\pgfpointxy{1150.0}{450.0}}
    \pgfusepath{stroke}
  \end{pgfscope}
  \begin{pgfscope}
    \pgfsetlinewidth{0.50mm}
    \pgfpathmoveto{\pgfpointxy{1200.0}{0.0}}
    \pgfpathlineto{\pgfpointxy{1200.0}{450.0}}
    \pgfusepath{stroke}
  \end{pgfscope}
  \begin{pgfscope}
    \pgfpathmoveto{\pgfpointxy{750.0}{100.0}}
    \pgfpathlineto{\pgfpointxy{1200.0}{100.0}}
    \pgfusepath{stroke}
  \end{pgfscope}
  \begin{pgfscope}
    \pgfsetlinewidth{0.50mm}
    \pgfpathmoveto{\pgfpointxy{750.0}{150.0}}
    \pgfpathlineto{\pgfpointxy{1200.0}{150.0}}
    \pgfusepath{stroke}
  \end{pgfscope}
  \begin{pgfscope}
    \pgfpathmoveto{\pgfpointxy{750.0}{200.0}}
    \pgfpathlineto{\pgfpointxy{1200.0}{200.0}}
    \pgfusepath{stroke}
  \end{pgfscope}
  \begin{pgfscope}
    \pgfpathmoveto{\pgfpointxy{750.0}{250.0}}
    \pgfpathlineto{\pgfpointxy{1200.0}{250.0}}
    \pgfusepath{stroke}
  \end{pgfscope}
  \begin{pgfscope}
    \pgfsetlinewidth{0.50mm}
    \pgfpathmoveto{\pgfpointxy{750.0}{300.0}}
    \pgfpathlineto{\pgfpointxy{1200.0}{300.0}}
    \pgfusepath{stroke}
  \end{pgfscope}
  \begin{pgfscope}
    \pgfpathmoveto{\pgfpointxy{750.0}{350.0}}
    \pgfpathlineto{\pgfpointxy{1200.0}{350.0}}
    \pgfusepath{stroke}
  \end{pgfscope}
  \begin{pgfscope}
    \pgfpathmoveto{\pgfpointxy{750.0}{400.0}}
    \pgfpathlineto{\pgfpointxy{1200.0}{400.0}}
    \pgfusepath{stroke}
  \end{pgfscope}
  \begin{pgfscope}
    \pgfsetlinewidth{0.50mm}
    \pgfpathmoveto{\pgfpointxy{750.0}{450.0}}
    \pgfpathlineto{\pgfpointxy{1200.0}{450.0}}
    \pgfusepath{stroke}
  \end{pgfscope}
  \begin{pgfscope}
    \pgfpathmoveto{\pgfpointxy{475.0}{225.0}}
    \pgfpathlineto{\pgfpointxy{725.0}{225.0}}
    \pgfusepath{stroke}
  \end{pgfscope}
  \begin{pgfscope}
    \pgfpathmoveto{\pgfpointxy{702.5}{231.029}}
    \pgfpatharcaxes{240.0}{270.0}{\pgfpointxy{45.0}{0.0}}{\pgfpointxy{0.0}{45.0}}
    \pgfusepath{stroke}
  \end{pgfscope}
  \begin{pgfscope}
    \pgfpathmoveto{\pgfpointxy{725.0}{225.0}}
    \pgfpatharcaxes{90.0}{120.0}{\pgfpointxy{45.0}{0.0}}{\pgfpointxy{0.0}{45.0}}
    \pgfusepath{stroke}
  \end{pgfscope}
  \pgftext[at={\pgfpointxy{25.0}{425.0}}]{4}
  \pgftext[at={\pgfpointxy{425.0}{425.0}}]{8}
  \pgftext[at={\pgfpointxy{1175.0}{425.0}}]{4}
  \pgftext[at={\pgfpointxy{1175.0}{25.0}}]{8}
  \pgftext[at={\pgfpointxy{25.0}{375.0}}]{6}
  \pgftext[at={\pgfpointxy{75.0}{375.0}}]{8}
  \pgftext[at={\pgfpointxy{1125.0}{425.0}}]{6}
  \pgftext[at={\pgfpointxy{1125.0}{375.0}}]{8}
  \pgftext[at={\pgfpointxy{25.0}{325.0}}]{3}
  \pgftext[at={\pgfpointxy{75.0}{325.0}}]{2}
  \pgftext[at={\pgfpointxy{125.0}{325.0}}]{1}
  \pgftext[at={\pgfpointxy{275.0}{325.0}}]{9}
  \pgftext[at={\pgfpointxy{325.0}{325.0}}]{4}
  \pgftext[at={\pgfpointxy{425.0}{325.0}}]{7}
  \pgftext[at={\pgfpointxy{1075.0}{425.0}}]{3}
  \pgftext[at={\pgfpointxy{1075.0}{375.0}}]{2}
  \pgftext[at={\pgfpointxy{1075.0}{325.0}}]{1}
  \pgftext[at={\pgfpointxy{1075.0}{175.0}}]{9}
  \pgftext[at={\pgfpointxy{1075.0}{125.0}}]{4}
  \pgftext[at={\pgfpointxy{1075.0}{25.0}}]{7}
  \pgftext[at={\pgfpointxy{275.0}{275.0}}]{1}
  \pgftext[at={\pgfpointxy{325.0}{275.0}}]{2}
  \pgftext[at={\pgfpointxy{1025.0}{175.0}}]{1}
  \pgftext[at={\pgfpointxy{1025.0}{125.0}}]{2}
  \pgftext[at={\pgfpointxy{175.0}{225.0}}]{8}
  \pgftext[at={\pgfpointxy{225.0}{225.0}}]{5}
  \pgftext[at={\pgfpointxy{275.0}{225.0}}]{7}
  \pgftext[at={\pgfpointxy{975.0}{275.0}}]{8}
  \pgftext[at={\pgfpointxy{975.0}{225.0}}]{5}
  \pgftext[at={\pgfpointxy{975.0}{175.0}}]{7}
  \pgftext[at={\pgfpointxy{125.0}{175.0}}]{3}
  \pgftext[at={\pgfpointxy{175.0}{175.0}}]{4}
  \pgftext[at={\pgfpointxy{925.0}{325.0}}]{3}
  \pgftext[at={\pgfpointxy{925.0}{275.0}}]{4}
  \pgftext[at={\pgfpointxy{25.0}{125.0}}]{7}
  \pgftext[at={\pgfpointxy{125.0}{125.0}}]{4}
  \pgftext[at={\pgfpointxy{325.0}{125.0}}]{1}
  \pgftext[at={\pgfpointxy{875.0}{425.0}}]{7}
  \pgftext[at={\pgfpointxy{875.0}{325.0}}]{4}
  \pgftext[at={\pgfpointxy{875.0}{125.0}}]{1}
  \pgftext[at={\pgfpointxy{325.0}{75.0}}]{5}
  \pgftext[at={\pgfpointxy{375.0}{75.0}}]{2}
  \pgftext[at={\pgfpointxy{825.0}{125.0}}]{5}
  \pgftext[at={\pgfpointxy{825.0}{75.0}}]{2}
  \pgftext[at={\pgfpointxy{25.0}{25.0}}]{8}
  \pgftext[at={\pgfpointxy{325.0}{25.0}}]{7}
  \pgftext[at={\pgfpointxy{375.0}{25.0}}]{3}
  \pgftext[at={\pgfpointxy{425.0}{25.0}}]{9}
  \pgftext[at={\pgfpointxy{775.0}{425.0}}]{8}
  \pgftext[at={\pgfpointxy{775.0}{125.0}}]{7}
  \pgftext[at={\pgfpointxy{775.0}{75.0}}]{3}
  \pgftext[at={\pgfpointxy{775.0}{25.0}}]{9}
  \pgftext[bottom,at={\pgfpointxy{600.0}{237.0}}]{clockwise}
  \pgftext[top,at={\pgfpointxy{600.0}{213.0}}]{$90^\circ$-rotation}
\end{pgfpicture}
  \caption{A $90^\circ$-rotation of a Sudoku grid}
  \label{sudoku.fig.rot}
\end{figure}

Another way to transform a Sudoku grid into a Sudoku grid is to permute its rows or columns. Now, we have to be
careful here: while every permutation of rows and columns preserves row and column conditions, some permutations
may violate box conditions. For example, after swapping the columns 3 and 7 in the Sudoku grid in Fig.~\ref{sudoku.fig.swap1}
we end up with a grid which is no longer a Sudoku because the box condition is violated in box~7.

Therefore, only those permutations of rows and columns which do not violate box conditions can be allowed.
Such permutations of rows and columns will be referred to as \emph{shuffles} and will be defined formally below.

\begin{figure}
  \centering
  \input pic/swap1.pgf
  \caption{Violating a box condition}
  \label{sudoku.fig.swap1}
\end{figure}

Finally, if we systematically replace digits in a Sudoku grid by other digits
taking care that each digit is always replaced by the same digit and that different digits are
replaced by different digits, we will end up with another Sudoku grid.
Such transformations will be referred to as \emph{relabellings}.

%
\bgroup

\def\1{5}
\def\2{1}
\def\3{7}
\def\4{9}
\def\5{8}
\def\6{3}
\def\7{2}
\def\8{4}
\def\9{6}
\def\smlrelabelling{%
$\left(\begin{smallmatrix}
  1 & 2 & 3 & 4 & 5 & 6 & 7 & 8 & 9\\
  \1 & \2 & \3 & \4 & \5 & \6 & \7 & \8 & \9
\end{smallmatrix}\right)$
}

More succinctly, each relabelling reduces to applying a fixed permutation of digits to each digit in a Sudoku grid.
For example, Fig.~\ref{sudoku.fig.relabelling} illustrates the relabelling of a Sudoku grid by applying the permutation
$$
  \begin{pmatrix}
    1 & 2 & 3 & 4 & 5 & 6 & 7 & 8 & 9\\
    \1 & \2 & \3 & \4 & \5 & \6 & \7 & \8 & \9
  \end{pmatrix}.
$$
The concise notation above conveys the information that each 1 in the grid should be changed to~5, each~2
should be changed to~1 and so on.

\begin{figure}
  \centering
\begin{pgfpicture}
  \pgfsetxvec{\pgfpoint{\acadpgfunit}{0pt}}
  \pgfsetyvec{\pgfpoint{0pt}{\acadpgfunit}}
  \pgfsetlinewidth{\acadpgflinewidth}
  \pgftransformshift{\pgfpointxy{150.0}{75.0}}

  \begin{pgfscope}
    \pgfsetlinewidth{0.50mm}
    \pgfpathmoveto{\pgfpointxy{-50.0}{0.0}}
    \pgfpathlineto{\pgfpointxy{-50.0}{450.0}}
    \pgfusepath{stroke}
  \end{pgfscope}
  \begin{pgfscope}
    \pgfpathmoveto{\pgfpointxy{0.0}{0.0}}
    \pgfpathlineto{\pgfpointxy{0.0}{450.0}}
    \pgfusepath{stroke}
  \end{pgfscope}
  \begin{pgfscope}
    \pgfpathmoveto{\pgfpointxy{-50.0}{50.0}}
    \pgfpathlineto{\pgfpointxy{400.0}{50.0}}
    \pgfusepath{stroke}
  \end{pgfscope}
  \begin{pgfscope}
    \pgfsetlinewidth{0.50mm}
    \pgfpathmoveto{\pgfpointxy{-50.0}{0.0}}
    \pgfpathlineto{\pgfpointxy{400.0}{0.0}}
    \pgfusepath{stroke}
  \end{pgfscope}
  \begin{pgfscope}
    \pgfpathmoveto{\pgfpointxy{50.0}{0.0}}
    \pgfpathlineto{\pgfpointxy{50.0}{450.0}}
    \pgfusepath{stroke}
  \end{pgfscope}
  \begin{pgfscope}
    \pgfsetlinewidth{0.50mm}
    \pgfpathmoveto{\pgfpointxy{100.0}{0.0}}
    \pgfpathlineto{\pgfpointxy{100.0}{450.0}}
    \pgfusepath{stroke}
  \end{pgfscope}
  \begin{pgfscope}
    \pgfpathmoveto{\pgfpointxy{150.0}{0.0}}
    \pgfpathlineto{\pgfpointxy{150.0}{450.0}}
    \pgfusepath{stroke}
  \end{pgfscope}
  \begin{pgfscope}
    \pgfpathmoveto{\pgfpointxy{200.0}{0.0}}
    \pgfpathlineto{\pgfpointxy{200.0}{450.0}}
    \pgfusepath{stroke}
  \end{pgfscope}
  \begin{pgfscope}
    \pgfsetlinewidth{0.50mm}
    \pgfpathmoveto{\pgfpointxy{250.0}{0.0}}
    \pgfpathlineto{\pgfpointxy{250.0}{450.0}}
    \pgfusepath{stroke}
  \end{pgfscope}
  \begin{pgfscope}
    \pgfpathmoveto{\pgfpointxy{300.0}{0.0}}
    \pgfpathlineto{\pgfpointxy{300.0}{450.0}}
    \pgfusepath{stroke}
  \end{pgfscope}
  \begin{pgfscope}
    \pgfpathmoveto{\pgfpointxy{350.0}{0.0}}
    \pgfpathlineto{\pgfpointxy{350.0}{450.0}}
    \pgfusepath{stroke}
  \end{pgfscope}
  \begin{pgfscope}
    \pgfsetlinewidth{0.50mm}
    \pgfpathmoveto{\pgfpointxy{400.0}{0.0}}
    \pgfpathlineto{\pgfpointxy{400.0}{450.0}}
    \pgfusepath{stroke}
  \end{pgfscope}
  \begin{pgfscope}
    \pgfpathmoveto{\pgfpointxy{-50.0}{100.0}}
    \pgfpathlineto{\pgfpointxy{400.0}{100.0}}
    \pgfusepath{stroke}
  \end{pgfscope}
  \begin{pgfscope}
    \pgfsetlinewidth{0.50mm}
    \pgfpathmoveto{\pgfpointxy{-50.0}{150.0}}
    \pgfpathlineto{\pgfpointxy{400.0}{150.0}}
    \pgfusepath{stroke}
  \end{pgfscope}
  \begin{pgfscope}
    \pgfpathmoveto{\pgfpointxy{-50.0}{200.0}}
    \pgfpathlineto{\pgfpointxy{400.0}{200.0}}
    \pgfusepath{stroke}
  \end{pgfscope}
  \begin{pgfscope}
    \pgfpathmoveto{\pgfpointxy{-50.0}{250.0}}
    \pgfpathlineto{\pgfpointxy{400.0}{250.0}}
    \pgfusepath{stroke}
  \end{pgfscope}
  \begin{pgfscope}
    \pgfsetlinewidth{0.50mm}
    \pgfpathmoveto{\pgfpointxy{-50.0}{300.0}}
    \pgfpathlineto{\pgfpointxy{400.0}{300.0}}
    \pgfusepath{stroke}
  \end{pgfscope}
  \begin{pgfscope}
    \pgfpathmoveto{\pgfpointxy{-50.0}{350.0}}
    \pgfpathlineto{\pgfpointxy{400.0}{350.0}}
    \pgfusepath{stroke}
  \end{pgfscope}
  \begin{pgfscope}
    \pgfpathmoveto{\pgfpointxy{-50.0}{400.0}}
    \pgfpathlineto{\pgfpointxy{400.0}{400.0}}
    \pgfusepath{stroke}
  \end{pgfscope}
  \begin{pgfscope}
    \pgfsetlinewidth{0.50mm}
    \pgfpathmoveto{\pgfpointxy{-50.0}{450.0}}
    \pgfpathlineto{\pgfpointxy{400.0}{450.0}}
    \pgfusepath{stroke}
  \end{pgfscope}
  \begin{pgfscope}
    \pgfpathmoveto{\pgfpointxy{425.0}{225.0}}
    \pgfpathlineto{\pgfpointxy{775.0}{225.0}}
    \pgfusepath{stroke}
  \end{pgfscope}
  \begin{pgfscope}
    \pgfpathmoveto{\pgfpointxy{752.5}{231.029}}
    \pgfpatharcaxes{240.0}{270.0}{\pgfpointxy{45.0}{0.0}}{\pgfpointxy{0.0}{45.0}}
    \pgfusepath{stroke}
  \end{pgfscope}
  \begin{pgfscope}
    \pgfpathmoveto{\pgfpointxy{775.0}{225.0}}
    \pgfpatharcaxes{90.0}{120.0}{\pgfpointxy{45.0}{0.0}}{\pgfpointxy{0.0}{45.0}}
    \pgfusepath{stroke}
  \end{pgfscope}
  \begin{pgfscope}
    \pgfsetlinewidth{0.50mm}
    \pgfpathmoveto{\pgfpointxy{800.0}{0.0}}
    \pgfpathlineto{\pgfpointxy{800.0}{450.0}}
    \pgfusepath{stroke}
  \end{pgfscope}
  \begin{pgfscope}
    \pgfpathmoveto{\pgfpointxy{850.0}{0.0}}
    \pgfpathlineto{\pgfpointxy{850.0}{450.0}}
    \pgfusepath{stroke}
  \end{pgfscope}
  \begin{pgfscope}
    \pgfpathmoveto{\pgfpointxy{800.0}{50.0}}
    \pgfpathlineto{\pgfpointxy{1250.0}{50.0}}
    \pgfusepath{stroke}
  \end{pgfscope}
  \begin{pgfscope}
    \pgfsetlinewidth{0.50mm}
    \pgfpathmoveto{\pgfpointxy{800.0}{0.0}}
    \pgfpathlineto{\pgfpointxy{1250.0}{0.0}}
    \pgfusepath{stroke}
  \end{pgfscope}
  \begin{pgfscope}
    \pgfpathmoveto{\pgfpointxy{900.0}{0.0}}
    \pgfpathlineto{\pgfpointxy{900.0}{450.0}}
    \pgfusepath{stroke}
  \end{pgfscope}
  \begin{pgfscope}
    \pgfsetlinewidth{0.50mm}
    \pgfpathmoveto{\pgfpointxy{950.0}{0.0}}
    \pgfpathlineto{\pgfpointxy{950.0}{450.0}}
    \pgfusepath{stroke}
  \end{pgfscope}
  \begin{pgfscope}
    \pgfpathmoveto{\pgfpointxy{1000.0}{0.0}}
    \pgfpathlineto{\pgfpointxy{1000.0}{450.0}}
    \pgfusepath{stroke}
  \end{pgfscope}
  \begin{pgfscope}
    \pgfpathmoveto{\pgfpointxy{1050.0}{0.0}}
    \pgfpathlineto{\pgfpointxy{1050.0}{450.0}}
    \pgfusepath{stroke}
  \end{pgfscope}
  \begin{pgfscope}
    \pgfsetlinewidth{0.50mm}
    \pgfpathmoveto{\pgfpointxy{1100.0}{0.0}}
    \pgfpathlineto{\pgfpointxy{1100.0}{450.0}}
    \pgfusepath{stroke}
  \end{pgfscope}
  \begin{pgfscope}
    \pgfpathmoveto{\pgfpointxy{1150.0}{0.0}}
    \pgfpathlineto{\pgfpointxy{1150.0}{450.0}}
    \pgfusepath{stroke}
  \end{pgfscope}
  \begin{pgfscope}
    \pgfpathmoveto{\pgfpointxy{1200.0}{0.0}}
    \pgfpathlineto{\pgfpointxy{1200.0}{450.0}}
    \pgfusepath{stroke}
  \end{pgfscope}
  \begin{pgfscope}
    \pgfsetlinewidth{0.50mm}
    \pgfpathmoveto{\pgfpointxy{1250.0}{0.0}}
    \pgfpathlineto{\pgfpointxy{1250.0}{450.0}}
    \pgfusepath{stroke}
  \end{pgfscope}
  \begin{pgfscope}
    \pgfpathmoveto{\pgfpointxy{800.0}{100.0}}
    \pgfpathlineto{\pgfpointxy{1250.0}{100.0}}
    \pgfusepath{stroke}
  \end{pgfscope}
  \begin{pgfscope}
    \pgfsetlinewidth{0.50mm}
    \pgfpathmoveto{\pgfpointxy{800.0}{150.0}}
    \pgfpathlineto{\pgfpointxy{1250.0}{150.0}}
    \pgfusepath{stroke}
  \end{pgfscope}
  \begin{pgfscope}
    \pgfpathmoveto{\pgfpointxy{800.0}{200.0}}
    \pgfpathlineto{\pgfpointxy{1250.0}{200.0}}
    \pgfusepath{stroke}
  \end{pgfscope}
  \begin{pgfscope}
    \pgfpathmoveto{\pgfpointxy{800.0}{250.0}}
    \pgfpathlineto{\pgfpointxy{1250.0}{250.0}}
    \pgfusepath{stroke}
  \end{pgfscope}
  \begin{pgfscope}
    \pgfsetlinewidth{0.50mm}
    \pgfpathmoveto{\pgfpointxy{800.0}{300.0}}
    \pgfpathlineto{\pgfpointxy{1250.0}{300.0}}
    \pgfusepath{stroke}
  \end{pgfscope}
  \begin{pgfscope}
    \pgfpathmoveto{\pgfpointxy{800.0}{350.0}}
    \pgfpathlineto{\pgfpointxy{1250.0}{350.0}}
    \pgfusepath{stroke}
  \end{pgfscope}
  \begin{pgfscope}
    \pgfpathmoveto{\pgfpointxy{800.0}{400.0}}
    \pgfpathlineto{\pgfpointxy{1250.0}{400.0}}
    \pgfusepath{stroke}
  \end{pgfscope}
  \begin{pgfscope}
    \pgfsetlinewidth{0.50mm}
    \pgfpathmoveto{\pgfpointxy{800.0}{450.0}}
    \pgfpathlineto{\pgfpointxy{1250.0}{450.0}}
    \pgfusepath{stroke}
  \end{pgfscope}
  \pgftext[at={\pgfpointxy{-25.0}{425.0}}]{4}
  \pgftext[at={\pgfpointxy{375.0}{425.0}}]{8}
  \pgftext[at={\pgfpointxy{-25.0}{375.0}}]{6}
  \pgftext[at={\pgfpointxy{25.0}{375.0}}]{8}
  \pgftext[at={\pgfpointxy{-25.0}{325.0}}]{3}
  \pgftext[at={\pgfpointxy{25.0}{325.0}}]{2}
  \pgftext[at={\pgfpointxy{75.0}{325.0}}]{1}
  \pgftext[at={\pgfpointxy{225.0}{325.0}}]{9}
  \pgftext[at={\pgfpointxy{275.0}{325.0}}]{4}
  \pgftext[at={\pgfpointxy{375.0}{325.0}}]{7}
  \pgftext[at={\pgfpointxy{225.0}{275.0}}]{1}
  \pgftext[at={\pgfpointxy{275.0}{275.0}}]{2}
  \pgftext[at={\pgfpointxy{125.0}{225.0}}]{8}
  \pgftext[at={\pgfpointxy{175.0}{225.0}}]{5}
  \pgftext[at={\pgfpointxy{225.0}{225.0}}]{7}
  \pgftext[at={\pgfpointxy{75.0}{175.0}}]{3}
  \pgftext[at={\pgfpointxy{125.0}{175.0}}]{4}
  \pgftext[at={\pgfpointxy{-25.0}{125.0}}]{7}
  \pgftext[at={\pgfpointxy{75.0}{125.0}}]{4}
  \pgftext[at={\pgfpointxy{275.0}{125.0}}]{1}
  \pgftext[at={\pgfpointxy{275.0}{75.0}}]{5}
  \pgftext[at={\pgfpointxy{325.0}{75.0}}]{2}
  \pgftext[at={\pgfpointxy{-25.0}{25.0}}]{8}
  \pgftext[at={\pgfpointxy{275.0}{25.0}}]{7}
  \pgftext[at={\pgfpointxy{325.0}{25.0}}]{3}
  \pgftext[at={\pgfpointxy{375.0}{25.0}}]{9}
  \pgftext[bottom,at={\pgfpointxy{600.0}{237.0}}]{relabelling}
  \pgftext[at={\pgfpointxy{825.0}{425.0}}]{\4}
  \pgftext[at={\pgfpointxy{1225.0}{425.0}}]{\8}
  \pgftext[at={\pgfpointxy{825.0}{375.0}}]{\6}
  \pgftext[at={\pgfpointxy{875.0}{375.0}}]{\8}
  \pgftext[at={\pgfpointxy{825.0}{325.0}}]{\3}
  \pgftext[at={\pgfpointxy{875.0}{325.0}}]{\2}
  \pgftext[at={\pgfpointxy{925.0}{325.0}}]{\1}
  \pgftext[at={\pgfpointxy{1075.0}{325.0}}]{\9}
  \pgftext[at={\pgfpointxy{1125.0}{325.0}}]{\4}
  \pgftext[at={\pgfpointxy{1225.0}{325.0}}]{\7}
  \pgftext[at={\pgfpointxy{1075.0}{275.0}}]{\1}
  \pgftext[at={\pgfpointxy{1125.0}{275.0}}]{\2}
  \pgftext[at={\pgfpointxy{975.0}{225.0}}]{\8}
  \pgftext[at={\pgfpointxy{1025.0}{225.0}}]{\5}
  \pgftext[at={\pgfpointxy{1075.0}{225.0}}]{\7}
  \pgftext[at={\pgfpointxy{925.0}{175.0}}]{\3}
  \pgftext[at={\pgfpointxy{975.0}{175.0}}]{\4}
  \pgftext[at={\pgfpointxy{825.0}{125.0}}]{\7}
  \pgftext[at={\pgfpointxy{925.0}{125.0}}]{\4}
  \pgftext[at={\pgfpointxy{1125.0}{125.0}}]{\1}
  \pgftext[at={\pgfpointxy{1125.0}{75.0}}]{\5}
  \pgftext[at={\pgfpointxy{1175.0}{75.0}}]{\2}
  \pgftext[at={\pgfpointxy{825.0}{25.0}}]{\8}
  \pgftext[at={\pgfpointxy{1125.0}{25.0}}]{\7}
  \pgftext[at={\pgfpointxy{1175.0}{25.0}}]{\3}
  \pgftext[at={\pgfpointxy{1225.0}{25.0}}]{\9}
  \pgftext[top,at={\pgfpointxy{600.0}{213.0}}]{\smlrelabelling}
\end{pgfpicture}
  \caption{Relabelling}
  \label{sudoku.fig.relabelling}
\end{figure}

\egroup
%

\section{Sudoku transformations}
\label{sudoku.sec.transf-def}

As we shall see, defining relabellings and shuffles formally is easy.
In contrast to that, defining geometric Sudoku transformations can be challenging. To save time
and space in both defining and formally reasoning about them,
we shall define geometric transformations as chains of reflections.

It is quite usual in many daily activities to take an object and apply
to it a sequence of transformations. In modern mathematics a transformation is modelled by a function
between two sets of objects, and then a sequence of transformations
$$
  A \overset f \longrightarrow B \overset g \longrightarrow C
$$
is usually denoted by $g \circ f$. Here's why. When this chain of transformations
acts on an element $x \in A$ we usually write:
$$
  g(f(x))
$$
because we first apply $f$ to $x$ to obtain $f(x)$, and then apply $g$ to $f(x)$. Thus:
$$
  (g \circ f)(x) = g(f(x)).
$$

The \emph{identity transformation $\id$} is the Sudoku transformation that changes nothing.
It has the property that $\id(\calA) = \calA$ for every grid $\calA$.
We also say that $\id$ is the \emph{trivial Sudoku transformation}.
A Sudoku transformation $\theta$ is \emph{nontrivial} if $\theta \ne \id$.

We shall now define geometric Sudoku transformations as chains of two kinds of reflections.

\begin{DEF}
  Let $\calA = [A_{ij}]_{9 \times 9}$ and $\calB = [B_{ij}]_{9 \times 9}$ be Sudoku grids.
  \begin{itemize}
  \item
  The \emph{reflection about the main diagonal} is a Sudoku transformation $\delta$ such that
  $\delta(\calA) = \calB$ just in case $B_{ij} = A_{ji}$, for all $1 \le i, j  \le 9$.
  \item
  The \emph{reflection about the horizontal axis} is a Sudoku transformation $\chi$ such that
  $\chi(\calA) = \calB$ just in case $B_{ij} = A_{10-i,j}$, for all $1 \le i, j \le 9$.
\end{itemize}
\end{DEF}

\begin{DEF}\label{sudoku.def.geom}
  This is the list of \emph{geometric Sudoku transformations:}
  \begin{itemize}\itemsep -1pt
    \item the identity transformation $\id$;
    \item the reflections $\delta$ and $\chi$ defined above;
    \item the \emph{reflection about the vertical axis} $\overline\chi = \delta\circ\chi\circ\delta$;
    \item the \emph{reflection about the auxiliary diagonal} $\overline\delta = \overline\chi\circ\delta\circ\overline\chi$;
    \item the \emph{clockwise $90^\circ$-rotation} $\rho_1 = \delta\circ\chi$;
    \item the \emph{clockwise $180^\circ$-rotation} $\rho_2 = \rho_1\circ\rho_1$; and
    \item the \emph{clockwise $270^\circ$-rotation} $\rho_3 = \rho_1\circ\rho_1\circ\rho_1$.
  \end{itemize}
\end{DEF}

Recall that only well-behaved row and column permutations preserve the quality of being a Sudoku grid.
To isolate those from the bad ones we shall introduce a few auxiliary notions. Let
$$
  \Pi = \big\{ \{1, 2, 3\}, \; \{4, 5, 6\}, \; \{7, 8, 9\} \big\}
$$
be a \emph{fundamental partition} of the index set $\{1, 2, \ldots, 9\}$. The elements of $\Pi$ will be
referred to as \emph{blocks} of~$\Pi$. So, $\{4, 5, 6\}$ is a block of~$\Pi$, while $\{1, 3, 5\}$ is not.
A bijection $f : \{1, 2, \ldots, 9\} \to \{1, 2, \ldots, 9\}$ \emph{respects $\Pi$} if the following holds:
\begin{itemize}
  \item if $i$ and $j$ belong to the same block of $\Pi$, then so do $f(i)$ and $f(j)$; and
  \item if $i$ and $j$ belong to different blocks of $\Pi$, then so do $f(i)$ and $f(j)$.
\end{itemize}
The idea behind this is that a row/column permutation transforms a Sudoku grid into a Sudoku grid
just in case the way it operates on row/column indices respects~$\Pi$. This is how we distinguish
well-behaved row/column permutations. Let us now define the remaining Sudoku transformations.

\begin{DEF}\label{sudoku.def.shuf-relab}
  Let $\calA = [A_{ij}]_{9 \times 9}$ and $\calB = [B_{ij}]_{9 \times 9}$ be Sudoku grids
  and let $f : \{1, 2, \ldots, 9\} \to \{1, 2, \ldots, 9\}$ be a bijection which respects~$\Pi$.
  \begin{itemize}
  \item
    The \emph{$f$-shuffle of rows} is a Sudoku transformation $\sigma_f$ such that
    $\sigma_f(\calA) = \calB$ just in case $B_{ij} = A_{f(i)\,j}$, for all $1 \le i, j  \le 9$.
  \item
    The \emph{$f$-shuffle of columns} is a Sudoku transformation $\tau_f$ such that
    $\tau_f(\calA) = \calB$ just in case $B_{ij} = A_{i\,f(j)}$, for all $1 \le i, j \le 9$.
  \end{itemize}
\end{DEF}

\begin{DEF}\label{sudoku.def.relab}
  Let $\calA = [A_{ij}]_{9 \times 9}$ and $\calB = [B_{ij}]_{9 \times 9}$ be Sudoku grids
  and let $f : \{1, 2, \ldots, 9\} \to \{1, 2, \ldots, 9\}$ be a bijection.
  The \emph{$f$-relabelling} is a Sudoku transformation $\lambda_f$ such that
  $\lambda_f(\calA) = \calB$ just in case $B_{ij} = \{f(d) : d \in A_{ij} \}$, for all $1 \le i, j \le 9$.
\end{DEF}

\begin{DEF}
  An \emph{elementary Sudoku transformation} is any of the transformations listed in Definitions~\ref{sudoku.def.geom},
  \ref{sudoku.def.shuf-relab} and~\ref{sudoku.def.relab}.
\end{DEF}

Clearly, a sequence of elementary Sudoku transformations (such as rotate-shuffle-relabel)
is again a Sudoku transformation. But, is that all?
Is there any other way to transform a Sudoku grid that might have eluded our scrutiny?
This question was addressed by Adler and Adler in~\cite{AdlerAdler} and the answer is that there
are no other meaningful Sudoku transformations.

Adler and Adler start by analyzing full Sudoku grids. To facilitate the presentation of their main
result, let us introduce several new notions. By a \emph{Sudoku board} we mean the set of pairs
$$
  \calN = \{(i, j) : 1 \le i,j \le 9\}.
$$
These are the indices of cells in a Sudoku grid. Sudoku transformations are then introduced in terms of
partitions of $\calN$ with the following intuition. Let $N_1 \subseteq \calN$ be the set of all the
cells that contain the same Sudoku digit, not necessarily~1; let $N_2 \subseteq \calN$ be the set of all the
cells that contain the same Sudoku digit different from the one that determines $N_1$ (again, this digit is not required to be~2);
and so on until we form $N_1$, $N_2$, \ldots, $N_9$. Let us stress again that the cells in each $N_i$ have to contain
the same digit, but the cells in $N_i$ are not required to contain the digit $i$, $1 \le i \le 9$. Then
$\{N_1, N_2, \ldots, N_9\}$ is a partition of $\calN$, each block $N_i$ has size~9 and each $N_i$
intersects each row, column and box of the grid exactly once.

Following~\cite{AdlerAdler} we shall, therefore, say that a partition $\{M_1, M_2, \ldots, M_9\}$ of $\calN$ is \emph{valid} if
\begin{itemize}
  \item each $M_i$ has size 9, $1 \le i \le 9$;
  \item each $M_i$ intersects each row, column and box of the grid exactly once, $1 \le i \le 9$.
\end{itemize}

\begin{DEF} \cite{AdlerAdler}
  A \emph{fundamental transformation} is every bijection
  $$
    f : \calN \to \calN
  $$
  which takes valid partitions to valid partitions. More precisely, for every 
  valid partition $\{M_1, M_2, \ldots, M_9\}$ of $\calN$ we have that
  $\{f(M_1), f(M_2), \ldots, f(M_9)\}$ is also a valid partition of $\calN$
  (here, $f(M) = \{f(x) : x \in M\}$).
\end{DEF}

Note that fundamental transformations operate on \emph{full} Sudoku grids!

\begin{THM}[The characterization of fundamental transformations~\cite{AdlerAdler}]
  Every fundamental transformation is of the form
  $\theta_1 \circ \theta_2 \circ \ldots \circ \theta_n$
  where each $\theta_i$ is an elementary Sudoku transformation.
\end{THM}

This theorem justifies the following expansion of the notion of fundamental transformation to all
Sudoku grids, including those that are not necessarily full.

\begin{DEF}
  A \emph{transformation of a Sudoku grid}, or a \emph{Sudoku transformation} for short,
  is any chain $\theta_1 \circ \theta_2 \circ \ldots \circ \theta_n$ of elementary Sudoku transformations.
  A \emph{cell transformation} is any chain $\theta_1 \circ \theta_2 \circ \ldots \circ \theta_n$
  of geometric transformations and shuffles.
\end{DEF}

We can now address the issue of essentially different grids from the introduction:
we say that two Sudoku grids are \emph{essentially different} if no Sudoku transformation
takes one of them onto the other.

\begin{THM}[Normal form]\label{sudoku.thm.normalform}
  Every Sudoku transformation can be written as $\xi\circ\lambda$ where $\xi$ is a
  cell transformation and $\lambda$ is a relabelling.
\end{THM}
\begin{proof}
  Exercise~\ref{sudoku.zad.TL=LT}.
\end{proof}

In most cases when a nontrivial Sudoku transformation is applied to a Sudoku grid $\calA$ the outcome
will differ from $\calA$. However, those Sudoku transformations that leave a Sudoku grid unchanged are
of particular interest: these are the symmetries of the grid! In the combinatorial context such as Sudoku, symmetries are
often referred to as \emph{automorphisms}.

\begin{DEF}
  Let $\calA$ be a Sudoku grid and $\theta$ a Sudoku transformation. We say that $\theta$ is an \emph{automorphism} of $\calA$
  if $\theta(\calA) = \calA$. A Sudoku grid $\calA$ is \emph{symmetric} if $\theta(\calA) = \calA$ for some
  nontrivial automorphism~$\theta$.
\end{DEF}

\begin{EX}
  Fig.~\ref{sudoku.fig.invariant-grids}~$(a)$ is an example of a full Sudoku grid with a nontrivial automorphism $\tau_f \circ \sigma_f$
  where
  $$
    f = \begin{pmatrix}
      1 & 2 & 3 & 4 & 5 & 6 & 7 & 8 & 9\\
      4 & 5 & 6 & 7 & 8 & 9 & 1 & 2 & 3
    \end{pmatrix},
  $$
  while Fig.~\ref{sudoku.fig.invariant-grids}~$(b)$ is an example of a full Sudoku grid with a nontrivial automorphism
  $\rho_2 \circ \lambda_g$ where $g(i) = 10 - i$.
\end{EX}

\begin{figure}
  \centering
  \input pic/Ex3.pgf
  \caption{A full Sudoku grid invariant for: $(a)$ a combination of shuffles; $(b)$ a combination of a rotation and a relabelling}
  \label{sudoku.fig.invariant-grids}
\end{figure}

\begin{ZAD}
  $(a)$ Prove that $\theta\circ\theta = \id$ whenever $\theta$ is a reflection.

  $(b)$ Prove that $\overline\delta = \delta\circ\chi\circ\delta\circ\chi\circ\delta$.

  $(c)$ Prove that $\rho_2\circ\rho_2 = \id$.

  $(d)$ Prove that $\rho_3 = \chi\circ\delta$.

  $(e)$ Let $\gamma_1$, $\gamma_2$, \ldots, $\gamma_n$ be arbitrary geometric transformations. Prove that
  $\gamma_1 \circ \gamma_2 \circ \ldots \circ \gamma_n$ is again a geometric transformation.
\end{ZAD}

\begin{ZAD}
  Prove that $\chi$ is a shuffle of rows and $\overline\chi$ is a shuffle of columns.
  Write down the corresponding bijections and check that they respect~$\Pi$.
\end{ZAD}

\begin{ZAD}
  Prove that $\xi\circ\lambda = \lambda\circ\xi$ for every cell transformation $\xi$ and every
  relabelling~$\lambda$.
\end{ZAD}

\begin{ZAD}\label{sudoku.zad.TL=LT}
  Prove Theorem~\ref{sudoku.thm.normalform}.
\end{ZAD}

\begin{ZAD}
  $(a)$ Prove that $\id$ is an automorphism of every Sudoku grid.

  $(b)$ Prove: if $\theta_1$ and $\theta_2$ are automorphisms of a Sudoku grid $\calA$ then
  $\theta_1 \circ \theta_2$ is also an automorphism of $\calA$.
\end{ZAD}

\section{Gurth’s Symmetrical Placement}
\label{sudoku.sec.Gurth}

We have now reached the point where we can formulate and prove Gurth's Symmetrical Placement Theorem.

Recall that cell transformations actually move cells, so for a cell transformation $\xi$ we shall write
$\xi(r_i c_j) = r_\ell c_m$ to indicate that the cell $r_i c_j$ in the original grid is moved to the cell~$r_\ell c_m$.
For example, $\delta(r_1 c_2) = r_2 c_1$ and $\rho_1(r_1 c_2) = r_2 c_9$. Note that $\delta(r_i c_i) = r_i c_i$ for all $1 \le i \le 9$
and $\rho_1(r_5 c_5) = \rho_2(r_5 c_5) = \rho_3(r_5 c_5) = r_5 c_5$.

\begin{LEM}\label{sudoku.lem.transf-full-grid}
  $(a)$ A Sudoku transformation takes a full Sudoku grid to a full Sudoku grid.

  $(b)$ Let $\calA$ and $\calS$ be Sudoku grids and $\theta$ a Sudoku transformation. If $\calS$ is a solution of $\calA$
  then $\theta(\calS)$ is a solution of $\theta(\calA)$.
\end{LEM}
\begin{proof}
  Exercise~\ref{sudoku.zad.transf-full-grid}.
\end{proof}

\begin{THM}[Gurth’s Symmetrical Placement]\label{sudoku.thm.gurth}
  Let $\calA$ and $\calS$ be Sudoku grids. If $\theta$ is an automorphism of $\calA$ and $\calS$ is a unique solution of $\calA$
  then $\theta$ is an automorphism of $\calS$.
\end{THM}
\begin{proof}
  Assume that $\calS$ is a unique solution of $\calA$ and that $\theta$ is an automorphism of $\calA$. Then $\theta(\calA) = \calA$.
  Since $\calS$ is a solution of $\calA$, Lemma~\ref{sudoku.lem.transf-full-grid}~$(b)$ gives us that
  $\theta(\calS)$ is then a solution of $\theta(\calA) = \calA$. So, both $\calS$ and $\theta(\calS)$ are solutions of $\calA$.
  However, $\calS$ is a unique solution of $\calA$, so $\theta(\calS) = \calS$, and $\theta$ is an automorphism of $\calS$.
\end{proof}

As an immediate corollary of Theorem~\ref{sudoku.thm.deduc=>uniq} we now have:

\begin{COR}[Gurth’s Symmetrical Placement -- Deducible version]
  Let $\calA$ and $\calS$ be Sudoku grids. If $\theta$ is an automorphism of $\calA$ and $\calS$ is a logically deducible solution of $\calA$
  then $\theta$ is an automorphism of $\calS$.
\end{COR}

\begin{EX}
  Let us now demonstrate Gurth’s Symmetrical Placement in action. 
  In the Cracking the Cryptic's video ``The Video They Said We Should NEVER
  Release''\footnote{https://www.youtube.com/watch?v=iiQ9MXrfb7A} at some point (timestamp 25:48)
  Simon reaches the position in~Fig.~\ref{sudoku.fig.gurth-ex} (with pencilmarks irrelevant for the discussion omitted; also
  it is safe to ignore the blue line along the auxiliary diagonal).

  \emph{Assume that the puzzle has a unique solution} and let us show that $r_9c_9$ cannot be~9.
  Suppose, to the contrary, that $r_9c_9$ contains the digit~9. Then 
  $$
    \theta = \overline \delta \circ \lambda_f \text{\quad where\quad} f = \begin{pmatrix}
      1 & 2 & 3 & 4 & 5 & 6 & 7 & 8 & 9\\
      1 & 2 & 3 & 5 & 4 & 7 & 6 & 9 & 8
    \end{pmatrix}
  $$
  is an automorphism of the grid with 9 filled in $r_9c_9$.
  Gurth's Symmetrical Placement Theorem then ensures that $\theta$ is an automorphism of
  the solution of the puzzle. Now, look at the cell $r_6c_4$ whose only
  options are 8 or~9. Since $\overline \delta$ takes $r_6c_4$ to itself and $\theta$ is an automorphism
  of the solution of the puzzle, it follows that $\lambda_f(8) = 8$ in case $r_6c_4$ contains 8 in the solved puzzle,
  or $\lambda_f(9) = 9$ in case $r_6c_4$ contains 9 in the solved puzzle. But this is not possible because
  $\lambda_f(8) = 9$ and $\lambda_f(9) = 8$. Contradiction! Therefore, $r_9c_9$ cannot be~9.
\end{EX}

\begin{figure}
  \centering
  \def\f#1{{\Large\textbf{#1}}}
  \def\a#1#2{$\substack{#1 \; #2}$}
\begin{pgfpicture}
  \pgfsetxvec{\pgfpoint{\acadpgfunit}{0pt}}
  \pgfsetyvec{\pgfpoint{0pt}{\acadpgfunit}}
  \pgfsetlinewidth{\acadpgflinewidth}
  \pgftransformshift{\pgfpointxy{25.0}{50.0}}

  \begin{pgfscope}
    \pgfsetlinewidth{0.50mm}
    \pgfpathmoveto{\pgfpointxy{0.0}{0.0}}
    \pgfpathlineto{\pgfpointxy{0.0}{900.0}}
    \pgfusepath{stroke}
  \end{pgfscope}
  \begin{pgfscope}
    \pgfsetlinewidth{0.50mm}
    \pgfpathmoveto{\pgfpointxy{0.0}{900.0}}
    \pgfpathlineto{\pgfpointxy{900.0}{900.0}}
    \pgfusepath{stroke}
  \end{pgfscope}
  \begin{pgfscope}
    \pgfpathmoveto{\pgfpointxy{100.0}{0.0}}
    \pgfpathlineto{\pgfpointxy{100.0}{900.0}}
    \pgfusepath{stroke}
  \end{pgfscope}
  \begin{pgfscope}
    \pgfpathmoveto{\pgfpointxy{200.0}{0.0}}
    \pgfpathlineto{\pgfpointxy{200.0}{900.0}}
    \pgfusepath{stroke}
  \end{pgfscope}
  \begin{pgfscope}
    \pgfsetlinewidth{0.50mm}
    \pgfpathmoveto{\pgfpointxy{300.0}{0.0}}
    \pgfpathlineto{\pgfpointxy{300.0}{900.0}}
    \pgfusepath{stroke}
  \end{pgfscope}
  \begin{pgfscope}
    \pgfpathmoveto{\pgfpointxy{0.0}{800.0}}
    \pgfpathlineto{\pgfpointxy{900.0}{800.0}}
    \pgfusepath{stroke}
  \end{pgfscope}
  \begin{pgfscope}
    \pgfpathmoveto{\pgfpointxy{400.0}{0.0}}
    \pgfpathlineto{\pgfpointxy{400.0}{900.0}}
    \pgfusepath{stroke}
  \end{pgfscope}
  \begin{pgfscope}
    \pgfpathmoveto{\pgfpointxy{500.0}{0.0}}
    \pgfpathlineto{\pgfpointxy{500.0}{900.0}}
    \pgfusepath{stroke}
  \end{pgfscope}
  \begin{pgfscope}
    \pgfsetlinewidth{0.50mm}
    \pgfpathmoveto{\pgfpointxy{600.0}{0.0}}
    \pgfpathlineto{\pgfpointxy{600.0}{900.0}}
    \pgfusepath{stroke}
  \end{pgfscope}
  \begin{pgfscope}
    \pgfpathmoveto{\pgfpointxy{700.0}{0.0}}
    \pgfpathlineto{\pgfpointxy{700.0}{900.0}}
    \pgfusepath{stroke}
  \end{pgfscope}
  \begin{pgfscope}
    \pgfpathmoveto{\pgfpointxy{800.0}{0.0}}
    \pgfpathlineto{\pgfpointxy{800.0}{900.0}}
    \pgfusepath{stroke}
  \end{pgfscope}
  \begin{pgfscope}
    \pgfsetlinewidth{0.50mm}
    \pgfpathmoveto{\pgfpointxy{900.0}{0.0}}
    \pgfpathlineto{\pgfpointxy{900.0}{900.0}}
    \pgfusepath{stroke}
  \end{pgfscope}
  \begin{pgfscope}
    \pgfpathmoveto{\pgfpointxy{0.0}{700.0}}
    \pgfpathlineto{\pgfpointxy{900.0}{700.0}}
    \pgfusepath{stroke}
  \end{pgfscope}
  \begin{pgfscope}
    \pgfsetlinewidth{0.50mm}
    \pgfpathmoveto{\pgfpointxy{0.0}{600.0}}
    \pgfpathlineto{\pgfpointxy{900.0}{600.0}}
    \pgfusepath{stroke}
  \end{pgfscope}
  \begin{pgfscope}
    \pgfpathmoveto{\pgfpointxy{0.0}{500.0}}
    \pgfpathlineto{\pgfpointxy{900.0}{500.0}}
    \pgfusepath{stroke}
  \end{pgfscope}
  \begin{pgfscope}
    \pgfpathmoveto{\pgfpointxy{0.0}{400.0}}
    \pgfpathlineto{\pgfpointxy{900.0}{400.0}}
    \pgfusepath{stroke}
  \end{pgfscope}
  \begin{pgfscope}
    \pgfsetlinewidth{0.50mm}
    \pgfpathmoveto{\pgfpointxy{0.0}{300.0}}
    \pgfpathlineto{\pgfpointxy{900.0}{300.0}}
    \pgfusepath{stroke}
  \end{pgfscope}
  \begin{pgfscope}
    \pgfpathmoveto{\pgfpointxy{0.0}{200.0}}
    \pgfpathlineto{\pgfpointxy{900.0}{200.0}}
    \pgfusepath{stroke}
  \end{pgfscope}
  \begin{pgfscope}
    \pgfpathmoveto{\pgfpointxy{0.0}{100.0}}
    \pgfpathlineto{\pgfpointxy{900.0}{100.0}}
    \pgfusepath{stroke}
  \end{pgfscope}
  \begin{pgfscope}
    \pgfsetlinewidth{0.50mm}
    \pgfpathmoveto{\pgfpointxy{0.0}{0.0}}
    \pgfpathlineto{\pgfpointxy{900.0}{0.0}}
    \pgfusepath{stroke}
  \end{pgfscope}
  \begin{pgfscope}
    \pgfsetstrokecolor{acad2004aci140}
    \pgfpathmoveto{\pgfpointxy{900.0}{900.0}}
    \pgfpathlineto{\pgfpointxy{0.0}{0.0}}
    \pgfusepath{stroke}
  \end{pgfscope}
  \pgftext[at={\pgfpointxy{50.0}{850.0}}]{\f8}
  \pgftext[at={\pgfpointxy{250.0}{850.0}}]{\f3}
  \pgftext[at={\pgfpointxy{350.0}{850.0}}]{\f7}
  \pgftext[at={\pgfpointxy{450.0}{850.0}}]{\f6}
  \pgftext[at={\pgfpointxy{150.0}{750.0}}]{\f2}
  \pgftext[at={\pgfpointxy{650.0}{750.0}}]{\f6}
  \pgftext[at={\pgfpointxy{50.0}{650.0}}]{\f1}
  \pgftext[at={\pgfpointxy{350.0}{650.0}}]{\f5}
  \pgftext[at={\pgfpointxy{750.0}{650.0}}]{\f7}
  \pgftext[at={\pgfpointxy{50.0}{550.0}}]{\f5}
  \pgftext[at={\pgfpointxy{250.0}{550.0}}]{\f8}
  \pgftext[at={\pgfpointxy{50.0}{450.0}}]{\f9}
  \pgftext[at={\pgfpointxy{850.0}{450.0}}]{\f7}
  \pgftext[at={\pgfpointxy{850.0}{350.0}}]{\f6}
  \pgftext[at={\pgfpointxy{650.0}{350.0}}]{\f4}
  \pgftext[at={\pgfpointxy{150.0}{250.0}}]{\f8}
  \pgftext[at={\pgfpointxy{250.0}{150.0}}]{\f9}
  \pgftext[at={\pgfpointxy{550.0}{250.0}}]{\f9}
  \pgftext[at={\pgfpointxy{550.0}{50.0}}]{\f4}
  \pgftext[at={\pgfpointxy{450.0}{50.0}}]{\f8}
  \pgftext[at={\pgfpointxy{850.0}{250.0}}]{\f3}
  \pgftext[at={\pgfpointxy{750.0}{150.0}}]{\f2}
  \pgftext[at={\pgfpointxy{650.0}{50.0}}]{\f1}
  \pgftext[at={\pgfpointxy{350.0}{350.0}}]{\a89}
  \pgftext[at={\pgfpointxy{850.0}{50.0}}]{\a59}
\end{pgfpicture}
  \caption{Application of Gurth’s Symmetrical Placement Theorem to eliminate 9 from $r_9c_9$}
  \label{sudoku.fig.gurth-ex}
\end{figure}

Let us now prove a simple corollary of Gurth’s Symmetrical Placement Theorem.

\begin{LEM}\label{sudoku.lem.geomaut}
  Let $\calS$ be a full Sudoku grid and let $\theta$ be an automorphism of $\calS$. If $\theta$ is a geometric transformation
  then $\theta = \id$.
\end{LEM}
\begin{proof}
  Assume that $\theta$ is a geometric transformation. Then it is easy to see that
  $\theta$ takes every cell in box~5 to some cell in box~5 (Exercise~\ref{sudoku.zad.box5}).
  Since $\theta$ makes no relabellings (by definition) it follows that $\theta$ takes every cell in box~5 to itself,
  and the only geometric transformation that fixes all the cells in box~5 is the identity. Therefore, $\theta = \id$.
\end{proof}

A \emph{geometric automorphism} of a Sudoku grid is any geometric transformation which is at an automorphism of the grid.
 
\begin{THM}\label{sudoku.thm.geomaut}
  Assume that a Sudoku grid $\calA$ has a deducible (or, equivalently, unique) solution.
  Then $\id$ is the only geometric automorphism of~$\calA$.
\end{THM}
\begin{proof}
  Let $\calS$ be a deducible solution of $\calA$ and let $\theta$ be a geometric automorphism of $\calA$.
  Theorem~\ref{sudoku.thm.gurth} then tells us that $\theta$ is an automorphism of $\calS$,
  so Lemma~\ref{sudoku.lem.geomaut} ensures that $\theta = \id$.
\end{proof}

Finally, let us weaken Gurth's Symmetrical Placement Theorem and thus justify the
practice of ``reasoning by symmetry''. Namely, if a conclusion can be reached using the
information contained in some part of the grid, then the ``symmetrical'' conclusion follows
immediately if we are given ``symmetrical'' information in the ``symmetrical'' part of the grid.
In order to make this feeling precise, we have to be able to apply Sudoku transformations to
Sudoku formulas. Note that this weaker result \emph{does not assume uniqueness of the solution!}

Let $\theta$ be a Sudoku transformation. By Theorem~\ref{sudoku.thm.normalform} there is a cell transformation $\xi$ and a relabelling
$\lambda$ such that $\theta = \xi \circ \lambda$. We have referred to this decomposition as the \emph{normal form} of the Sudoku
transformation.

\begin{DEF}\label{sudoku.def.theta-alpha}
  Let $\theta = \xi \circ \lambda$ be a Sudoku transformation in its normal form, where $\xi$ is a cell transformation and $\lambda$ is a relabelling.
  Then for a Sudoku formula  $\phi$ we define $\theta(\phi)$ as follows:
  \begin{itemize}\itemsep -1pt
    \item $\theta(\top) = \top$ and $\theta(\bot) = \bot$;
    \item if $\phi = r_i c_j \not\approx d$ then $\theta(\phi) = \xi(r_i c_j) \not\approx \lambda(d)$;
    \item if $\phi = \lnot \alpha$ then $\theta(\phi) = \lnot \theta(\alpha)$;
    \item if $\phi = \alpha \star \beta$, where $\star$ is one of the connectives $\land$, $\lor$, $\Rightarrow$, $\Leftrightarrow$
          then $\theta(\phi) = \theta(\alpha) \star \theta(\beta)$.
  \end{itemize}
  For a set of Sudoku formulas $\Phi$ let $\theta(\Phi) = \{\theta(\phi) : \phi \in \Phi\}$.
\end{DEF}

\begin{LEM}\label{sudoku.lem.thetaathetaalpha}
  Let $\calA$ be a Sudoku grid, $\alpha$ a Sudoku formula and $\theta$ a Sudoku transformation. Then
  
  $(a)$ $\calA \models \alpha$ if and only if $\theta(\calA) \models \theta(\alpha)$;

  $(b)$ $\theta(\Cns(\calA)) = \Cns(\theta(\calA))$.
\end{LEM}
\begin{proof}
  Exercise~\ref{sudoku.zad.thetaathetaalpha}.
\end{proof}

\begin{LEM}\label{sudoku.lem.theta-ax}
  If $\alpha$ is a Sudoku axiom and $\theta$ is a Sudoku transformation then $\theta(\alpha)$ is also a Sudoku axiom.
\end{LEM}
\begin{proof}
  We have to show that $\theta(\alpha)$ holds in every full Sudoku grid. So, let $\calS$ be a full Sudoku grid and
  let $\theta^{-1}$ be the inverse of $\theta$. Then $\theta^{-1}(\calS)$ is a full Sudoku grid (Lemma~\ref{sudoku.zad.transf-full-grid}),
  so $\theta^{-1}(\calS) \models \alpha$ because $\alpha$ is a Sudoku axiom. But then Lemma~\ref{sudoku.lem.thetaathetaalpha}
  yields that $\theta(\theta^{-1}(\calS)) \models \theta(\alpha)$, that is,
  $\calS \models \theta(\alpha)$ because $\theta \circ \theta^{-1} = \id$.
\end{proof}

\begin{THM}[``Symmetric reasoning'']
  Let $\Phi$ be a set of Sudoku formulas, $\alpha$ a Sudoku formula and $\theta$ a Sudoku transformation.
  If $\Phi \vdash \alpha$ then $\theta(\Phi) \vdash \theta(\alpha)$.
\end{THM}
\begin{proof}
  Assume that $\Phi \vdash \alpha$. Then there is a proof in Sudoku logic
  $$
    \phi_1, \; \phi_2, \; \ldots, \; \phi_n
  $$
  of $\alpha = \phi_n$ from premisses in $\Phi$. Let us show that 
  \begin{equation}\label{sudoku.eq.proof-theta-phi}
    \theta(\phi_1), \; \theta(\phi_2), \; \ldots, \; \theta(\phi_n)
  \end{equation}
  is a proof of $\theta(\alpha)$ from premisses in $\theta(\Phi)$. Clearly, $\theta(\phi_n) = \theta(\alpha)$ because $\phi_n = \alpha$, so
  the main job here is to prove that \eqref{sudoku.eq.proof-theta-phi} is actually a proof. And this is straightforward:
  \begin{itemize}
    \item if $\phi_i$ is a Sudoku axiom, then $\theta(\phi_i)$ is also a Sudoku axiom by Lemma~\ref{sudoku.lem.theta-ax};
    \item if $\phi_i \in \Phi$ is a premise then $\theta(\phi_i) \in \theta(\Phi)$ trivially;
    \item if $\phi_i$ is obtained by modus ponens from $\phi_p$ and $\phi_q$ for some $p, q < i$ then without loss of generality
          we can assume that $\phi_q$ is of the form $\phi_p \Rightarrow \phi_i$, so
          $$
            \theta(\phi_q) = \theta(\phi_p \Rightarrow \phi_i) = (\theta(\phi_p) \Rightarrow \theta(\phi_i))
          $$
          by Definition~\ref{sudoku.def.theta-alpha}, whence follows that $\theta(\phi_i)$ can be obtained by modus ponens from
          $\theta(\phi_p)$ and $\theta(\phi_q)$.
  \end{itemize}
  This completes the proof.
\end{proof}

\clearpage

\begin{ZAD}\label{sudoku.zad.transf-full-grid}
  Prove Lemma~\ref{sudoku.lem.transf-full-grid}.
\end{ZAD}

\begin{ZAD}
  Prove that the following Sudoku grid \emph{does not} have a deducible solution:
  \begin{center}
\begin{pgfpicture}
  \pgfsetxvec{\pgfpoint{\acadpgfunit}{0pt}}
  \pgfsetyvec{\pgfpoint{0pt}{\acadpgfunit}}
  \pgfsetlinewidth{\acadpgflinewidth}
  \pgftransformshift{\pgfpointxy{50.0}{62.5}}

  \begin{pgfscope}
    \pgfsetlinewidth{0.50mm}
    \pgfpathmoveto{\pgfpointxy{0.0}{0.0}}
    \pgfpathlineto{\pgfpointxy{0.0}{450.0}}
    \pgfusepath{stroke}
  \end{pgfscope}
  \begin{pgfscope}
    \pgfpathmoveto{\pgfpointxy{50.0}{0.0}}
    \pgfpathlineto{\pgfpointxy{50.0}{450.0}}
    \pgfusepath{stroke}
  \end{pgfscope}
  \begin{pgfscope}
    \pgfpathmoveto{\pgfpointxy{0.0}{50.0}}
    \pgfpathlineto{\pgfpointxy{450.0}{50.0}}
    \pgfusepath{stroke}
  \end{pgfscope}
  \begin{pgfscope}
    \pgfsetlinewidth{0.50mm}
    \pgfpathmoveto{\pgfpointxy{0.0}{0.0}}
    \pgfpathlineto{\pgfpointxy{450.0}{0.0}}
    \pgfusepath{stroke}
  \end{pgfscope}
  \begin{pgfscope}
    \pgfpathmoveto{\pgfpointxy{100.0}{0.0}}
    \pgfpathlineto{\pgfpointxy{100.0}{450.0}}
    \pgfusepath{stroke}
  \end{pgfscope}
  \begin{pgfscope}
    \pgfsetlinewidth{0.50mm}
    \pgfpathmoveto{\pgfpointxy{150.0}{0.0}}
    \pgfpathlineto{\pgfpointxy{150.0}{450.0}}
    \pgfusepath{stroke}
  \end{pgfscope}
  \begin{pgfscope}
    \pgfpathmoveto{\pgfpointxy{200.0}{0.0}}
    \pgfpathlineto{\pgfpointxy{200.0}{450.0}}
    \pgfusepath{stroke}
  \end{pgfscope}
  \begin{pgfscope}
    \pgfpathmoveto{\pgfpointxy{250.0}{0.0}}
    \pgfpathlineto{\pgfpointxy{250.0}{450.0}}
    \pgfusepath{stroke}
  \end{pgfscope}
  \begin{pgfscope}
    \pgfsetlinewidth{0.50mm}
    \pgfpathmoveto{\pgfpointxy{300.0}{0.0}}
    \pgfpathlineto{\pgfpointxy{300.0}{450.0}}
    \pgfusepath{stroke}
  \end{pgfscope}
  \begin{pgfscope}
    \pgfpathmoveto{\pgfpointxy{350.0}{0.0}}
    \pgfpathlineto{\pgfpointxy{350.0}{450.0}}
    \pgfusepath{stroke}
  \end{pgfscope}
  \begin{pgfscope}
    \pgfpathmoveto{\pgfpointxy{400.0}{0.0}}
    \pgfpathlineto{\pgfpointxy{400.0}{450.0}}
    \pgfusepath{stroke}
  \end{pgfscope}
  \begin{pgfscope}
    \pgfsetlinewidth{0.50mm}
    \pgfpathmoveto{\pgfpointxy{450.0}{0.0}}
    \pgfpathlineto{\pgfpointxy{450.0}{450.0}}
    \pgfusepath{stroke}
  \end{pgfscope}
  \begin{pgfscope}
    \pgfpathmoveto{\pgfpointxy{0.0}{100.0}}
    \pgfpathlineto{\pgfpointxy{450.0}{100.0}}
    \pgfusepath{stroke}
  \end{pgfscope}
  \begin{pgfscope}
    \pgfsetlinewidth{0.50mm}
    \pgfpathmoveto{\pgfpointxy{0.0}{150.0}}
    \pgfpathlineto{\pgfpointxy{450.0}{150.0}}
    \pgfusepath{stroke}
  \end{pgfscope}
  \begin{pgfscope}
    \pgfpathmoveto{\pgfpointxy{0.0}{200.0}}
    \pgfpathlineto{\pgfpointxy{450.0}{200.0}}
    \pgfusepath{stroke}
  \end{pgfscope}
  \begin{pgfscope}
    \pgfpathmoveto{\pgfpointxy{0.0}{250.0}}
    \pgfpathlineto{\pgfpointxy{450.0}{250.0}}
    \pgfusepath{stroke}
  \end{pgfscope}
  \begin{pgfscope}
    \pgfsetlinewidth{0.50mm}
    \pgfpathmoveto{\pgfpointxy{0.0}{300.0}}
    \pgfpathlineto{\pgfpointxy{450.0}{300.0}}
    \pgfusepath{stroke}
  \end{pgfscope}
  \begin{pgfscope}
    \pgfpathmoveto{\pgfpointxy{0.0}{350.0}}
    \pgfpathlineto{\pgfpointxy{450.0}{350.0}}
    \pgfusepath{stroke}
  \end{pgfscope}
  \begin{pgfscope}
    \pgfpathmoveto{\pgfpointxy{0.0}{400.0}}
    \pgfpathlineto{\pgfpointxy{450.0}{400.0}}
    \pgfusepath{stroke}
  \end{pgfscope}
  \begin{pgfscope}
    \pgfsetlinewidth{0.50mm}
    \pgfpathmoveto{\pgfpointxy{0.0}{450.0}}
    \pgfpathlineto{\pgfpointxy{450.0}{450.0}}
    \pgfusepath{stroke}
  \end{pgfscope}
  \pgftext[at={\pgfpointxy{75.0}{425.0}}]{2}
  \pgftext[at={\pgfpointxy{125.0}{425.0}}]{3}
  \pgftext[at={\pgfpointxy{425.0}{375.0}}]{2}
  \pgftext[at={\pgfpointxy{425.0}{325.0}}]{3}
  \pgftext[at={\pgfpointxy{25.0}{125.0}}]{3}
  \pgftext[at={\pgfpointxy{25.0}{75.0}}]{2}
  \pgftext[at={\pgfpointxy{325.0}{25.0}}]{3}
  \pgftext[at={\pgfpointxy{375.0}{25.0}}]{2}
  \pgftext[at={\pgfpointxy{125.0}{375.0}}]{1}
  \pgftext[at={\pgfpointxy{375.0}{325.0}}]{1}
  \pgftext[at={\pgfpointxy{325.0}{75.0}}]{1}
  \pgftext[at={\pgfpointxy{75.0}{125.0}}]{1}
  \pgftext[at={\pgfpointxy{125.0}{275.0}}]{6}
  \pgftext[at={\pgfpointxy{275.0}{325.0}}]{6}
  \pgftext[at={\pgfpointxy{325.0}{175.0}}]{6}
  \pgftext[at={\pgfpointxy{175.0}{125.0}}]{6}
  \pgftext[at={\pgfpointxy{125.0}{25.0}}]{7}
  \pgftext[at={\pgfpointxy{25.0}{325.0}}]{7}
  \pgftext[at={\pgfpointxy{325.0}{425.0}}]{7}
  \pgftext[at={\pgfpointxy{425.0}{125.0}}]{7}
  \pgftext[at={\pgfpointxy{125.0}{75.0}}]{5}
  \pgftext[at={\pgfpointxy{75.0}{325.0}}]{5}
  \pgftext[at={\pgfpointxy{325.0}{375.0}}]{5}
  \pgftext[at={\pgfpointxy{375.0}{125.0}}]{5}
  \pgftext[at={\pgfpointxy{225.0}{225.0}}]{6}
\end{pgfpicture}
  \end{center}
  (Hint: Use Theorem~\ref{sudoku.thm.geomaut}.)
\end{ZAD}

\begin{ZAD}
  Prove that in the following Sudoku grid\footnote{see https://www.youtube.com/watch?v=ac7Y97t2Wns}
  each cell on the main diagonal can contain only one of the digits 2, 4, 8:
  \begin{center}
\begin{pgfpicture}
  \pgfsetxvec{\pgfpoint{\acadpgfunit}{0pt}}
  \pgfsetyvec{\pgfpoint{0pt}{\acadpgfunit}}
  \pgfsetlinewidth{\acadpgflinewidth}
  \pgftransformshift{\pgfpointxy{12.5}{12.5}}

  \begin{pgfscope}
    \pgfsetlinewidth{0.50mm}
    \pgfpathmoveto{\pgfpointxy{0.0}{0.0}}
    \pgfpathlineto{\pgfpointxy{0.0}{450.0}}
    \pgfusepath{stroke}
  \end{pgfscope}
  \begin{pgfscope}
    \pgfpathmoveto{\pgfpointxy{50.0}{0.0}}
    \pgfpathlineto{\pgfpointxy{50.0}{450.0}}
    \pgfusepath{stroke}
  \end{pgfscope}
  \begin{pgfscope}
    \pgfpathmoveto{\pgfpointxy{0.0}{50.0}}
    \pgfpathlineto{\pgfpointxy{450.0}{50.0}}
    \pgfusepath{stroke}
  \end{pgfscope}
  \begin{pgfscope}
    \pgfsetlinewidth{0.50mm}
    \pgfpathmoveto{\pgfpointxy{0.0}{0.0}}
    \pgfpathlineto{\pgfpointxy{450.0}{0.0}}
    \pgfusepath{stroke}
  \end{pgfscope}
  \begin{pgfscope}
    \pgfpathmoveto{\pgfpointxy{100.0}{0.0}}
    \pgfpathlineto{\pgfpointxy{100.0}{450.0}}
    \pgfusepath{stroke}
  \end{pgfscope}
  \begin{pgfscope}
    \pgfsetlinewidth{0.50mm}
    \pgfpathmoveto{\pgfpointxy{150.0}{0.0}}
    \pgfpathlineto{\pgfpointxy{150.0}{450.0}}
    \pgfusepath{stroke}
  \end{pgfscope}
  \begin{pgfscope}
    \pgfpathmoveto{\pgfpointxy{200.0}{0.0}}
    \pgfpathlineto{\pgfpointxy{200.0}{450.0}}
    \pgfusepath{stroke}
  \end{pgfscope}
  \begin{pgfscope}
    \pgfpathmoveto{\pgfpointxy{250.0}{0.0}}
    \pgfpathlineto{\pgfpointxy{250.0}{450.0}}
    \pgfusepath{stroke}
  \end{pgfscope}
  \begin{pgfscope}
    \pgfsetlinewidth{0.50mm}
    \pgfpathmoveto{\pgfpointxy{300.0}{0.0}}
    \pgfpathlineto{\pgfpointxy{300.0}{450.0}}
    \pgfusepath{stroke}
  \end{pgfscope}
  \begin{pgfscope}
    \pgfpathmoveto{\pgfpointxy{350.0}{0.0}}
    \pgfpathlineto{\pgfpointxy{350.0}{450.0}}
    \pgfusepath{stroke}
  \end{pgfscope}
  \begin{pgfscope}
    \pgfpathmoveto{\pgfpointxy{400.0}{0.0}}
    \pgfpathlineto{\pgfpointxy{400.0}{450.0}}
    \pgfusepath{stroke}
  \end{pgfscope}
  \begin{pgfscope}
    \pgfsetlinewidth{0.50mm}
    \pgfpathmoveto{\pgfpointxy{450.0}{0.0}}
    \pgfpathlineto{\pgfpointxy{450.0}{450.0}}
    \pgfusepath{stroke}
  \end{pgfscope}
  \begin{pgfscope}
    \pgfpathmoveto{\pgfpointxy{0.0}{100.0}}
    \pgfpathlineto{\pgfpointxy{450.0}{100.0}}
    \pgfusepath{stroke}
  \end{pgfscope}
  \begin{pgfscope}
    \pgfsetlinewidth{0.50mm}
    \pgfpathmoveto{\pgfpointxy{0.0}{150.0}}
    \pgfpathlineto{\pgfpointxy{450.0}{150.0}}
    \pgfusepath{stroke}
  \end{pgfscope}
  \begin{pgfscope}
    \pgfpathmoveto{\pgfpointxy{0.0}{200.0}}
    \pgfpathlineto{\pgfpointxy{450.0}{200.0}}
    \pgfusepath{stroke}
  \end{pgfscope}
  \begin{pgfscope}
    \pgfpathmoveto{\pgfpointxy{0.0}{250.0}}
    \pgfpathlineto{\pgfpointxy{450.0}{250.0}}
    \pgfusepath{stroke}
  \end{pgfscope}
  \begin{pgfscope}
    \pgfsetlinewidth{0.50mm}
    \pgfpathmoveto{\pgfpointxy{0.0}{300.0}}
    \pgfpathlineto{\pgfpointxy{450.0}{300.0}}
    \pgfusepath{stroke}
  \end{pgfscope}
  \begin{pgfscope}
    \pgfpathmoveto{\pgfpointxy{0.0}{350.0}}
    \pgfpathlineto{\pgfpointxy{450.0}{350.0}}
    \pgfusepath{stroke}
  \end{pgfscope}
  \begin{pgfscope}
    \pgfpathmoveto{\pgfpointxy{0.0}{400.0}}
    \pgfpathlineto{\pgfpointxy{450.0}{400.0}}
    \pgfusepath{stroke}
  \end{pgfscope}
  \begin{pgfscope}
    \pgfsetlinewidth{0.50mm}
    \pgfpathmoveto{\pgfpointxy{0.0}{450.0}}
    \pgfpathlineto{\pgfpointxy{450.0}{450.0}}
    \pgfusepath{stroke}
  \end{pgfscope}
  \pgftext[at={\pgfpointxy{125.0}{375.0}}]{3}
  \pgftext[at={\pgfpointxy{425.0}{425.0}}]{2}
  \pgftext[at={\pgfpointxy{275.0}{425.0}}]{1}
  \pgftext[at={\pgfpointxy{75.0}{325.0}}]{5}
  \pgftext[at={\pgfpointxy{225.0}{325.0}}]{6}
  \pgftext[at={\pgfpointxy{375.0}{375.0}}]{4}
  \pgftext[at={\pgfpointxy{325.0}{325.0}}]{7}
  \pgftext[at={\pgfpointxy{125.0}{225.0}}]{7}
  \pgftext[at={\pgfpointxy{25.0}{175.0}}]{9}
  \pgftext[at={\pgfpointxy{375.0}{275.0}}]{7}
  \pgftext[at={\pgfpointxy{325.0}{225.0}}]{8}
  \pgftext[at={\pgfpointxy{425.0}{175.0}}]{1}
  \pgftext[at={\pgfpointxy{175.0}{275.0}}]{8}
  \pgftext[at={\pgfpointxy{275.0}{225.0}}]{3}
  \pgftext[at={\pgfpointxy{225.0}{175.0}}]{5}
  \pgftext[at={\pgfpointxy{125.0}{125.0}}]{6}
  \pgftext[at={\pgfpointxy{75.0}{75.0}}]{4}
  \pgftext[at={\pgfpointxy{25.0}{25.0}}]{2}
  \pgftext[at={\pgfpointxy{225.0}{125.0}}]{8}
  \pgftext[at={\pgfpointxy{175.0}{75.0}}]{6}
  \pgftext[at={\pgfpointxy{275.0}{25.0}}]{9}
  \pgftext[at={\pgfpointxy{325.0}{125.0}}]{2}
  \pgftext[at={\pgfpointxy{425.0}{75.0}}]{7}
  \pgftext[at={\pgfpointxy{375.0}{25.0}}]{6}
\end{pgfpicture}
  \end{center}
\end{ZAD}

\begin{ZAD}\label{sudoku.zad.box5}
  Prove that every geometric transformation takes every cell in box~5 to some cell in box~5.
\end{ZAD}

\begin{ZAD}
  $(a)$ Let $\theta = \xi \circ \lambda_f$ where $\xi$ is a reflection (see Definition~\ref{sudoku.def.geom}) and $\lambda_f$ is a relabelling.
  Then $f$ is an \emph{involution} (that is, $f \circ f = \id$). Moreover,
  there exist exactly three values $d \in \{1, 2, 3, 4, 5, 6, 7, 8, 9\}$ satisfying $f(d) = d$.
  (Hint: consider box~5).

  $(b)$ Let $\theta = \rho_i \circ \lambda_f$ where $i \in \{1, 2, 3\}$ (see Definition~\ref{sudoku.def.geom}) and $\lambda_f$ is a relabelling.
  Then there exists exactly one $d \in \{1, 2, 3, 4, 5, 6, 7, 8, 9\}$ satisfying $f(d) = d$.
  (Hint: consider box~5).
\end{ZAD}

\begin{ZAD}
  Let $\theta = \xi \circ \lambda_f$ where $\xi = \xi_1 \circ \ldots \circ \xi_n$ is a composition of several geometric transformations (see Definition~\ref{sudoku.def.geom}),
  and $\lambda_f$ is a relabelling.

  $(a)$ If there exists exactly one $d \in \{1, 2, 3, 4, 5, 6, 7, 8, 9\}$ satisfying $f(d) = d$ then $\xi$ is a rotation.

  $(b)$ If there exist exactly three values $d \in \{1, 2, 3, 4, 5, 6, 7, 8, 9\}$ satisfying $f(d) = d$ then $\xi$ is a reflection.

  $(c)$ If there exist more that three values $d \in \{1, 2, 3, 4, 5, 6, 7, 8, 9\}$ satisfying $f(d) = d$ then $\xi$ is the identity transformation.
\end{ZAD}

\begin{ZAD}\label{sudoku.zad.thetaathetaalpha}
  Prove Lemma~\ref{sudoku.lem.thetaathetaalpha}. (Hint: the implication from left to right in $(a)$ is a direct consequence of
  Definition~\ref{sudoku.def.theta-alpha}; the implication from right to left in $(a)$ follows from the former implication using the
  trick with $\theta^{-1}$ as in the proof of Lemma~\ref{sudoku.lem.theta-ax}; $(b)$ follows from $(a)$ straightforwardly.)
\end{ZAD}

\section{Concluding remarks}
\label{sudoku.sec.conclusion}

In this paper we have proposed a formal system of Sudoku logic that is sound and complete,
and is strong enough to support formal proofs of nontrivial Sudoku theorems such as
the Gurth's Symmetrical Placement Theorem, or the Corollary~\ref{sudoku.cor.quniq=deduc}
which claims that a Sudoku grid has a unique solution if and only if it has a deducible solution.

We shall now demonstrate that this system is so strong that it can even resolve the Uniqueness Controversy.
The Uniqueness Controversy revolves around a delicate question:
\begin{quote}
  \textbf{Uniqueness Controversy:}
  While searching for a deducible solution of a Sudoku,
  is one allowed to assume that the puzzle is uniquely solvable?
\end{quote}
There are solving techniques that rely on the assumption that the puzzle is uniquely solvable,
and in some cases using such techniques can significantly shorten the solving time. On the other hand,
many Sudoku puritans believe that assuming uniqueness
is unjustified unless we are explicitly told by the constructor
of the puzzle that it has a unique solution. Others, again, believe that
all logical deductions which assume that a puzzle has a unique solution can be found
by other means.

The system of Sudoku logic presented in this paper resolves the uniqueness assumption issue decisively:
once you are set off to find a \emph{logically deducible solution} of a Sudoku puzzle using Sudoku logic,
you can safely (and with no trepidation of heart) assume that the solution is unique. Namely,
the uniqueness assumption is \emph{not} an additional assumption, \emph{but an axiom of the Sudoku logic!}
Let us demonstrate this fact.

Let us write $r_i c_j \approx d$ as a shorthand for:
$$
  r_i c_j \approx d: \qquad \lnot(r_i c_j \not\approx d).
$$
The meaning of $r_i c_j \approx d$ is that ``the cell $r_i c_j$ could take value~$d$''.

For every full Sudoku grid $\calS = [S_{ij}]_{9 \times 9}$ where $S_{ij} = \{d_{ij}\}$, $1 \le i,j \le 9$,
let $\epsilon_\calS$ be the following formula:
$$
  \epsilon_\calS: \qquad \bigwedge_{1 \le i,j \le 9} r_i c_j \approx d_{ij}.
$$
If $\calA$ is a Sudoku grid, then $\calA \models \epsilon_\calS$ means ``$\calS$ could be a solution of $\calA$''.
Now, consider the following formula:
$$
  \Uniq: \qquad \bigwedge_{\calS' \ne \calS''} \lnot(\epsilon_{\calS'} \land \epsilon_{\calS''}),
$$
where the outer (big) conjunction ranges over all pairs of distinct full Sudoku grids $\calS'$ and $\calS''$.
The formula $\Uniq$ \emph{encodes uniqueness} -- it tells us that
for every pair of distinct full Sudoku grids $\calS'$ and $\calS''$,
it is not the case that both $\calS'$ and $\calS''$ could be solutions.

\begin{THM}
  $\Uniq$ is an axiom of the deductive system of Sudoku logic, that is, $\Uniq \in \Dax$.
\end{THM}
\begin{proof}
  Note that $\calS \models \Uniq$ holds for every full Sudoku grid $\calS$, and recall that
  $\Dax$ is the set of all Sudoku formulas that hold in every full Sudoku grid.
\end{proof}

Therefore,
  \textsl{if we accept the Sudoku logic as presented in this paper, $\Uniq$ is an axiom of the deductive system.
  Since axioms are always at our disposal in any formal deduction,
  there is \textbf{nothing controversial} in assuming uniqueness while searching for a logically deducible solution
  for a Sudoku grid!}

\begin{PROB}
  For the philosophically inclined reader the last conclusion is just a portal to new levels of controversy:
  \begin{itemize}
    \item Is this logical system too strong?
    \item Is it true that for every deduction in Sudoku logic where $\Uniq$ is used,
          there is another deduction of the same fact which does not use $\Uniq$ (``$\Uniq$-elimination'')?
    \item Is there a logical system for Sudoku which is sound and complete, and yet Uniqueness \emph{cannot} be logically deduced within the system?
  \end{itemize}
\end{PROB}

\appendix

\section*{Appendix}

\section{Completeness of Sudoku logic}
\label{sudoku.sec.app-A}

This entire section is devoted to proving Theorem~\ref{sudoku.thm.completeness} which establishes the completeness of Sudoku logic.
Since Sudoku logic is a straightforward modification of propositional logic, it comes as no surprise that the proof of
completeness of Sudoku logic is a straightforward modification the completeness proof of propositional logic
(see, for example,~\cite{epstein}). Let us now start developing the necessary machinery.

\begin{THM}[Semantic Deduction]\label{sudoku.thm.semded}
  Let $\Phi$ be a set of Sudoku propositions and let $\alpha$ and $\beta$ be Sudoku propositions.
  Then $\Phi \models \alpha \Rightarrow \beta$ if and only if $\Phi \cup \{\alpha\} \models \beta$.
\end{THM}
\begin{proof}
  Assume that $\Phi \cup \{\alpha\} \models \beta$. To show that $\Phi \models \alpha \Rightarrow \beta$
  take a Sudoku grid $\calA$ such that $\calA \models \Phi$ and let $\calS$ be a solution of $\calA$.
  By Lemma~\ref{sudoku.lem.SmodelsCnsA} we then have that $\calS \models \Phi$.
  To show that $\calS \models \alpha \Rightarrow \beta$ let us assume that
  $\calS \models \alpha$. Then $\calS \models \Phi \union \{\alpha\}$, so $\Phi \cup \{\alpha\} \models \beta$
  gives us that $\calS \models \beta$. The other direction is straightforward.
\end{proof}

\begin{LEM}\label{sudoku.lem.compl.1}
  If $\alpha$ is an arbitrary Sudoku formula and $\phi$ is a Sudoku axiom, then $\alpha \Rightarrow \phi$ is also
  a Sudoku axiom.
\end{LEM}
\begin{proof}
  Let $\calS$ be an arbitrary full Sudoku grid. Then $\calS \models \phi$ because $\phi$ is a Sudoku axiom. But then
  $\calS \models \alpha \Rightarrow \phi$: since the consequent $\phi$ is true in $\calS$, the implication
  $\alpha \Rightarrow \phi$ will be true in $\calS$ regardless of the status of the antecedent $\alpha$.
\end{proof}

\begin{THM}[Syntactic Deduction]\label{sudoku.thm.syntded}
  Let $\Phi$ be a set of Sudoku formulas and let $\alpha$ and $\beta$ be some Sudoku formulas. Then
  $\Phi \vdash \alpha \Rightarrow \beta$ if and only if $\Phi \union \{\alpha\} \vdash \beta$.
\end{THM}
\begin{proof}
  $(\Rightarrow)$ Assume that $\Phi \vdash \alpha \Rightarrow \beta$. Then there is a proof
  $$
    \phi_1, \; \ldots, \; \phi_{n-1}, \; \alpha \Rightarrow \beta
  $$
  where the premisses we need for the proof are taken from $\Phi$. But then
  $$
    \phi_1, \; \ldots, \; \phi_{n-1}, \; \alpha \Rightarrow \beta, \; \alpha, \; \beta
  $$
  is a proof of $\beta$ from the premisses $\Phi \union \{\alpha\}$. Note that in the last step we have applied modus ponens
  to the previous two formulas.

  $(\Leftarrow)$ Assume that $\Phi \union \{\alpha\} \vdash \beta$ and let
  $$
    \phi_1, \; \ldots, \; \phi_{n}
  $$
  be a proof of $\phi_n = \beta$ from the premisses in $\Phi \union \{\alpha\}$. Let us show then that
  $$
    \alpha \Rightarrow \phi_1, \; \ldots, \; \alpha \Rightarrow \phi_{n}
  $$
  is a proof from the premisses in $\Phi$. Since $\phi_n = \beta$ this will show that $\Phi \vdash \alpha \Rightarrow \beta$.
  We proceed by induction.

  For $i = 1$ we have that $\phi_1$ is an axiom, belongs to $\Phi$ or equals $\alpha$. If $\phi_1 \in \Dax$ then
  then $\alpha \Rightarrow \phi_1 \in \Dax$ by Lemma~\ref{sudoku.lem.compl.1}, so $\Phi \vdash \alpha \Rightarrow \phi_1$.
  If $\phi_1 \in \Phi$ then the following is a proof from the premisses in $\Phi$:
  $$
    \phi_1, \; \phi_1 \Rightarrow (\alpha \Rightarrow \phi_1), \; \alpha \Rightarrow \phi_1.
  $$
  Namely, the first formula belongs to $\Phi$, the second one is a tautology and hence an axiom, and the third is obtained by
  modus ponens. Finally, if $\phi_1 = \alpha$ then $\Phi \vdash \alpha \Rightarrow \alpha$
  because $\alpha \Rightarrow \alpha$ is a tautology, and hence a Sudoku axiom.

  Assume, now, that $\Phi \vdash \alpha \Rightarrow \phi_q$ for all $q < i$, and let us take a look at $\phi_i$.
  If $\phi_i$ is an axiom, belongs to $\Phi$ or equals $\alpha$ then $\Phi \vdash \phi_i \Rightarrow \alpha$ as in
  the base case $i = 1$. If this is not the case, then $\phi_i$ is obtained by modus ponens from some
  $\phi_p$ and $\phi_q = \phi_p \Rightarrow \phi_i$ where $p, q < i$. By the induction hypothesis we have that
  $\Phi \vdash \alpha \Rightarrow \phi_p$ and $\Phi \vdash \alpha \Rightarrow \phi_q$. But then we have that
  $$
  \begin{array}{rr@{\,}c@{\,}ll}
    1. & \Phi & \vdash & \alpha \Rightarrow \phi_p & [\text{induction hypothesis}]\\
    2. & \Phi & \vdash & \alpha \Rightarrow (\phi_p \Rightarrow \phi_i) & [\text{ind.\ hyp.\ and } \phi_q = \phi_p \Rightarrow \phi_i]\\
    3. &      & \vdash & \multicolumn{2}{@{}l@{}}{(\alpha \Rightarrow (\phi_p \Rightarrow \phi_i)) \Rightarrow ((\alpha \Rightarrow \phi_p) \Rightarrow (\alpha \Rightarrow \phi_i))}\\
    4. & \Phi & \vdash & (\alpha \Rightarrow \phi_p) \Rightarrow (\alpha \Rightarrow \phi_i) & [\text{modus ponens 2, 3}]\\
    5. & \Phi & \vdash & \alpha \Rightarrow \phi_i & [\text{modus ponens 1, 4}]
  \end{array}
  $$
  Note that the formula in step 3.\ is a tautology, and hence a Sudoku axiom.
  This completes the proof of the Syntactic Deduction Theorem.
\end{proof}

\begin{DEF}
  A set $\Phi$ of Sudoku formulas is \emph{complete} if for every Sudoku formula $\alpha$ we have that $\Phi \vdash \alpha$
  or $\Phi \vdash \lnot \alpha$.

  A set $\Phi$ of Sudoku formulas is \emph{inconsistent} if there is a Sudoku formula $\alpha$
  such that $\Phi \vdash \alpha$ and $\Phi \vdash \lnot\alpha$; and $\Phi$ is \emph{consistent}
  if it is not inconsistent.

  A set $\Phi$ of Sudoku formulas is a \emph{Sudoku theory} if it is \emph{deductively closed}, that is, if
  $\Phi \vdash \alpha$ then $\alpha \in \Phi$ for every Sudoku formula $\alpha$.

  The \emph{theory of a full Sudoku grid $\calS$} is the set $\Th(\calS)$ of all Sudoku formulas $\phi$
  such that $\calS \models \phi$.
\end{DEF}

\begin{LEM}\label{sudoku.lem.ThS}
  For every full Sudoku grid $\calS$ we have that $\Th(\calS)$ is a complete and consistent Sudoku theory.
\end{LEM}
\begin{proof}
  Let $\calS$ be a full Sudoku grid. Lemma~\ref{sudoku.lem.s-phi-alpha} ensures that $\Th(\calS)$ is
  deductively closed and hence a theory. It is also easy to see that $\Th(\calS)$ is a complete and consistent set of Sudoku formulas
  because we have a model in which we can check validity. Namely, for every Sudoku formula $\alpha$ we have that
  either $\calS \models \alpha$ or not $\calS \models \alpha$.
\end{proof}

\begin{LEM}\label{sudoku.lem.aux-compl}
  Let $\Phi$ be a set of Sudoku formulas and $\alpha$ a Sudoku formula.

  $(a)$ $\Phi$ is consistent if and only if every finite subset of $\Phi$ is consistent.

  $(b)$ $\Phi$ is inconsistent if and only if $\Phi \vdash \alpha$ for every Sudoku formula $\alpha$.

  $(c)$ $\Phi \nvdash \alpha$ if and only if $\Phi \union \{\lnot\alpha\}$ is consistent.
\end{LEM}
\begin{proof}
  $(a)$ If there is a finite subset of $\Phi$ which is inconsistent, then $\Phi$ is also inconsistent. Conversely,
  if $\Phi$ is inconsistent then, by definition, there is a Sudoku formula $\alpha$
  such that $\Phi \vdash \alpha$ and $\Phi \vdash \lnot\alpha$. This means that there are proofs
  $$
    \phi_1, \; \phi_2, \; \ldots, \phi_n \text{\quad and\quad} \psi_1, \; \psi_2, \; \ldots, \psi_m
  $$
  of $\phi_n = \alpha$ and $\psi_m = \lnot \alpha$ from the premisses in $\Phi$. Then it is easy to see that
  $$
    \Psi = \Phi \cap \{\phi_1, \phi_2, \ldots, \phi_n, \psi_1, \psi_2, \ldots, \psi_m\}
  $$
  is a finite inconsistent subset of $\Phi$.

  $(b)$
  If $\Phi \vdash \alpha$ for every Sudoku formula $\alpha$, then $\Phi \vdash \beta$ and 
  $\Phi \vdash \lnot\beta$ for some fixed $\beta$, so $\Phi$ is inconsistent by definition.
  Conversely, if $\Phi$ is inconsistent then by definition there is a formula $\beta$ such that
  $\Phi \vdash \beta$ and $\Phi \vdash \lnot\beta$. Let $\alpha$ be an arbitrary Sudoku formula. Then
  $$
  \begin{array}{rr@{\,}c@{\,}ll}
    1. & \Phi & \vdash & \beta & [\text{by assumption}]\\
    2. & \Phi & \vdash & \lnot\beta & [\text{by assumption}]\\
    3. &      & \vdash & \lnot\beta \Rightarrow (\beta \Rightarrow \alpha) & [\text{tautology}]\\
    4. & \Phi & \vdash & \beta \Rightarrow \alpha & [\text{modus ponens 2, 3}]\\
    5. & \Phi & \vdash & \alpha & [\text{modus ponens 1, 4}]
  \end{array}
  $$

  $(c)$ If $\Phi \vdash \alpha$ then $\Phi \union \{\lnot\alpha\}$ is inconsistent because we then have
  $\Phi \union \{\lnot\alpha\} \vdash \alpha$ and $\Phi \union \{\lnot\alpha\} \vdash \lnot\alpha$.
  Conversely, if $\Phi \union \{\lnot\alpha\}$ is inconsistent then by $(b)$ we have that 
  $\Phi \union \{\lnot\alpha\} \vdash \alpha$. Syntactic Deduction Theorem then gives us that
  $\Phi \vdash \lnot\alpha \Rightarrow \alpha$. Then:
  $$
  \begin{array}{rr@{\,}c@{\,}ll}
    1. & \Phi & \vdash & \lnot\alpha \Rightarrow \alpha & [\text{Syntactic Deduction Theorem}]\\
    2. &      & \vdash & \alpha \Rightarrow \alpha\\
    3. &      & \vdash & \multicolumn{2}{@{}l@{}}{(\lnot\alpha \Rightarrow \alpha) \Rightarrow ((\alpha \Rightarrow \alpha) \Rightarrow \alpha)}\\
    4. & \Phi & \vdash & (\alpha \Rightarrow \alpha) \Rightarrow \alpha & [\text{modus ponens 1, 3}]\\
    5. & \Phi & \vdash & \alpha & [\text{modus ponens 2, 4}]
  \end{array}
  $$
  Note that the formulas in steps 2 and 3 are tautologies, and hence Sudoku axioms. This completes the proof.
\end{proof}

\begin{LEM}\label{sudoku.lem.8}
  If $\alpha$ is a Sudoku formula such that $\nvdash \alpha$ then there is a complete and consistent Sudoku theory $\Phi$
  such that $\alpha \notin \Phi$.
\end{LEM}
\begin{proof}
  The set of all Sudoku formulas is countably infinite, so we can arrange all of them in an infinite sequence:
  $$
    \beta_0, \; \beta_1, \; \beta_2, \; \ldots
  $$
  Let us now construct a sequence of sets $\Phi_0, \Phi_1, \Phi_2, \ldots$ as follows. Put
  $$
    \Phi_0 = \{ \lnot\alpha \}
  $$
  and, assuming that $\Phi_n$ has been constructed, define $\Phi_{n+1}$ like this:
  $$
    \Phi_{n+1} = \begin{cases}
      \Phi_n \union \{\beta_n\}, & \text{if } \Phi_n \union \{\beta_n\} \text{ is consistent},\\
      \Phi_n, & \text{otherwise}.
    \end{cases}
  $$
  Finally, put $\Phi = \bigcup_{i \ge 0} \Phi_i$. Let us show that $\Phi$ is a complete and consistent Sudoku theory such that
  $\alpha \notin \Phi$.

  The consistency of $\Phi_0$ follows from Lemma~\ref{sudoku.lem.aux-compl}~$(c)$ (just put $\Phi = \0$ there).
  For all other $i$ the consistency of $\Phi_i$ follows by construction. So, each $\Phi_i$ is consistent.
  Now, if $\Phi$ is inconsistent then by Lemma~\ref{sudoku.lem.aux-compl}~$(a)$ there is a finite subset $\Psi$
  of $\Phi$ which is inconsistent. Since $\Psi$ is finite, there is some $n$ such that $\Phi \subseteq \Phi_n$.
  So, $\Phi_n$ is inconsistent by another application of Lemma~\ref{sudoku.lem.aux-compl}~$(a)$. Contradiction.
  This proves consistency of $\Phi$.

  To show that $\Phi$ is complete, take any Sudoku formula $\alpha$. If $\Phi \vdash \alpha$ we are done.
  If, however, $\Phi \nvdash \alpha$ then by Lemma~\ref{sudoku.lem.aux-compl}~$(c)$ we have that 
  $\Phi \union \{\lnot\alpha\}$ is consistent. Let $\lnot\alpha$ be the formula $\beta_j$ in our enumeration from the very
  beginning of the proof. Then it must be the case that $\Phi_j \union \{\beta_j\}$ is also consistent,
  so $\lnot\alpha = \beta_j \in \Phi_{j+1} \subseteq \Phi$ and thus $\Phi \vdash \lnot\alpha$.

  The same idea applies to show that $\Phi$ is a theory. Namely, we know that $\Phi$ is consistent, so
  if $\Phi \vdash \alpha$ then $\Phi \union \{\alpha\}$ is also consistent. Then
  the same argument as above ensures that $\alpha \in \Phi$. This completes the proof.
\end{proof}

\begin{LEM}\label{sudoku.lem.7}
  A set $\Phi$ of Sudoku formulas  is a complete and consistent Sudoku theory if and only if
  $\Phi = \Th(\calS)$ for some full Sudoku grid $\calS$.
\end{LEM}
\begin{proof}
  If $\Phi = \Th(\calS)$ for some full Sudoku grid $\calS$ then $\Phi$ is clearly a complete and consistent Sudoku theory
  since we have a models $\calS$ in which we can check validity of formulas.

  Conversely, assume that $\Phi$ is a complete and consistent Sudoku theory. Let us construct a full Sudoku grid $\calS$
  such that $\Phi = \Th(\calS)$. We proceed in several steps. Recall that $r_ic_j \approx d$ is a shorthand for
  $$
      r_ic_j \approx d: \qquad \lnot(r_i c_j \not\approx d).
  $$

  \medskip

  Step 1. Let us show that for each $i, j \in \{1, 2, \ldots, 9\}$ there is a unique $d \in \{1, 2, \ldots, 9\}$ such that $\Phi \vdash r_i c_j \approx d$.

  Proof. Fix an $i$ and $j$ and let us first show that such a $d$ exists. Suppose, to the contrary, that
  $\Phi \nvdash r_i c_j \approx 1$, $\Phi \nvdash r_i c_j \approx 2$, \ldots, $\Phi \nvdash r_i c_j \approx 9$.
  Then the completeness of $\Phi$ implies that
  $\Phi \vdash r_i c_j \not\approx 1$, $\Phi \vdash r_i c_j \not\approx 2$, \ldots, $\Phi \vdash r_i c_j \not\approx 9$, so
  $$
    \Phi \vdash r_i c_j \not\approx 1 \land r_i c_j \not\approx 2 \land \ldots \land r_i c_j \not\approx 9.
  $$
  On the other hand,
  $$
    \Phi \vdash r_i c_j \approx 1 \lor r_i c_j \approx 2 \lor \ldots \lor r_i c_j \approx 9,
  $$
  because this is a Sudoku axiom. Therefore, we have found a formula $\alpha$ such that $\Phi \vdash \alpha$ and
  $\Phi \vdash \lnot\alpha$, which contradicts the consistency of $\Phi$. This proves the existence part of the claim.

  The uniqueness part follows by similar reasoning. Suppose, to the contrary, that $\Phi \vdash r_i c_j \approx d$
  and $\Phi \vdash r_i c_j \approx e$ where $e \ne d$. Then
  $$
    \Phi \vdash r_i c_j \approx d \land r_i c_j \approx e.
  $$
  On the other hand, the fact that $e \ne d$ implies that
  $$
    \Phi \vdash \lnot(r_i c_j \approx d \land r_i c_j \approx e)
  $$
  is a Sudoku axiom, and we again get the contradiction with the consistency of $\Phi$. This completes Step~1.

  \medskip

  Step 2. Let $\calS = [\{d_{ij}\}]_{9 \times 9}$ be a grid where, for each $i, j \in \{1, 2, \ldots, 9\}$,
  we take $d_{ij}$ to be the unique digit such that $\Phi \vdash r_i c_j \approx d_{ij}$ whose existence is established in Step~1.
  Let us show that $\calS$ is a full Sudoku grid.

  Proof. This grid is clearly full (every cell has only one assigned value). Assume that $\calS$ is not a Sudoku grid.
  Then there exists a row, a column or a box in which some value is repeated. Without loss of generality, assume
  that there exist cells $r_i c_j$ and $r_i c_m$ such that $j \ne m$ but $d_{ij} = d_{im}$.
  Let $d$ be this common value. Then, by the construction,
  $$
    \Phi \vdash r_i c_j \approx d \land r_i c_m \approx d.
  $$
  On the other hand, the fact that $j \ne m$ implies that
  $$
    \Phi \vdash \lnot(r_i c_j \approx d \land r_i c_m \approx d)
  $$
  is a Sudoku axiom, and we get the contradiction with the consistency of $\Phi$. This completes Step~2.

  \medskip

  Step 3. $\Phi = \Th(\calS)$.

  Proof. Take any Sudoku formula $\alpha$ and let us show using induction on the number of logical connectives in $\alpha$ that 
  \begin{center}
    $\alpha \in \Th(\calS)$ if and only if $\alpha \in \Phi$.
  \end{center}
  Assume, first, that $\alpha$ has no logical connectives. Then $\alpha$ is a Sudoku proposition, that is,
  a statement of the form $r_i c_j \not\approx d$ for some $1 \le i, j, d \le 9$. If $\alpha \in \Th(\calS)$
  then $\calS \models r_i c_j \not\approx d$. Suppose that that $\alpha \notin \Phi$.
  Then $\Phi \nvdash \alpha$ because $\Phi$ is deductively closed, so $\Phi \vdash \lnot\alpha$ because
  $\Phi$ is complete. In other words, $\Phi \vdash r_i c_j \approx d$. By construction of $\calS$ (Step~1) we then have that
  $\calS \models r_i c_j \approx d$. Contradiction.

  On the other hand, if $\alpha \notin \Th(\calS)$ then it is not the case that $\calS \models \alpha$, so
  $\calS \models r_i c_j \approx d$. By construction of $\calS$ (Step~1) we then have that
  $\Phi \vdash r_i c_j \approx d$, whence $\Phi \nvdash r_i c_j \not\approx d$ because $\Phi$ is consistent.
  Therefore, $\alpha \notin \Phi$.

  Assume now that the statement
  \begin{equation}\label{sudoku.eq.ThS}
    \beta \in \Th(\calS) \text{\quad if and only if\quad} \beta \in \Phi
  \end{equation}
  holds for every formula $\beta$ with less logical connectives than $\alpha$.

  If $\alpha = \lnot \beta$ then $\calS \models \alpha$ means that
  $\calS \models \lnot\beta$, so it is not the case that $\calS \models \beta$, or, equivalently, $\beta \notin \Th(\calS)$.
  Since $\beta$ has less logical connectives than $\alpha$, the induction hypothesis \eqref{sudoku.eq.ThS} then yields that
  $\beta \notin \Phi$, so $\alpha = \lnot \beta \in \Phi$ because $\Phi$ is a complete theory.
  The other implication ($\alpha \in \Phi$ implies $\alpha in \Th(\calS)$) follows by similar reasoning.

  If $\alpha = \beta \land \gamma$ then $\calS \models \alpha$ means that
  $\calS \models \beta \land \gamma$. So, we have that $\calS \models \beta$ and $\calS \models \gamma$.
  Since both $\beta$ and $\gamma$ have less logical connectives than $\alpha$, the induction hypothesis \eqref{sudoku.eq.ThS} then yields that
  $\beta \in \Phi$ and $\gamma \in \Phi$. Therefore, $\Phi \vdash \beta$ and $\Phi \vdash \gamma$ because $\Phi$ is deductively closed,
  and then it easily follows that $\Phi \vdash \beta \land \gamma$. Another application of the fact that $\Phi$ is deductively
  closed yields that $\alpha = \beta \land \gamma \in \Phi$.
  The other implication ($\alpha \in \Phi$ implies $\alpha in \Th(\calS)$) follows by similar reasoning.

  There is no need to consider other possibilities for $\alpha$ because of the tautologies
  \begin{align*}
    \models &\; (\beta \lor \gamma) \Leftrightarrow \lnot (\lnot \beta \land \lnot \gamma),\\
    \models &\; (\beta \Rightarrow \gamma) \Leftrightarrow (\lnot \beta \lor \gamma) \text{\quad and}\\
    \models &\; (\beta \Leftrightarrow \gamma) \Leftrightarrow (\beta \Rightarrow \gamma) \land (\gamma \Rightarrow \beta).
  \end{align*}
  This completes the proof of Step~3 and the lemma.
\end{proof}

\begin{THM}[Completeness of axiomatization]\label{sudoku.thm.compl-ax}
  Let $\alpha$ be a Sudoku formula. Then: $\vdash \alpha$ if and only if $\models \alpha$.
\end{THM}
\begin{proof}
  Assume that $\vdash \alpha$. This means that there is a proof such that each formula in it is a Sudoku axiom or can be obtained by
  modus ponens. An easy induction on the length of the proof then ensures that for every full Sudoku grid $\calS$
  we have that $\calS \models \alpha$. This shows that $\models \alpha$.

  Assume, now, that $\nvdash \alpha$. Then by Lemma~\ref{sudoku.lem.8}
  there is a complete and consistent Sudoku theory $\Phi$ such that $\alpha \notin \Phi$.
  By Lemma~\ref{sudoku.lem.7} there is a full Sudoku grid $\calS$ such that $\Phi = \Th(\calS)$.
  Since $\alpha \notin \Phi$, it follows that it is not the case that $\calS \models \alpha$,
  so, by definition, $\not\models \alpha$.
\end{proof}

We are finally ready to prove Theorem~\ref{sudoku.thm.completeness}.

\begin{THMNONUM}[Completeness of Sudoku logic, Theorem~\ref{sudoku.thm.completeness}]
  Let $\Phi$ and $\Psi$ be two sets of Sudoku propositions. If $\Phi \models \Psi$ then
  $\Phi \vdash \Psi$.
\end{THMNONUM}
\begin{proof}
  Clearly, it suffices to show that $\Phi \models \alpha$ implies $\Phi \vdash \alpha$,
  for every set $\Phi$ of Sudoku propositions and every Sudoku proposition $\alpha$.
  Note, first, that $\Phi$ is a finite set as there are only finitely many Sudoku propositions (actually, one can form
  exactly $9^3$ distinct Sudoku propositions). So, let $\Phi = \{\phi_1, \phi_2, \ldots, \phi_n\}$ and assume that
  $$
    \{\phi_1, \phi_2, \ldots, \phi_n\} \models \alpha.
  $$
  Then several applications of the Semantic Deduction Theorem yield
  $$
     \models \phi_1 \Rightarrow (\phi_2 \Rightarrow (\ldots (\phi_n \Rightarrow \alpha))).
  $$
  By Theorem~\ref{sudoku.thm.compl-ax} we then get
  $$
    \vdash \phi_1 \Rightarrow (\phi_2 \Rightarrow (\ldots (\phi_n \Rightarrow \alpha))),
  $$
  and the Syntactic Deduction Theorem then yields
  $$
    \{\phi_1, \phi_2, \ldots, \phi_n\} \vdash \alpha.
  $$
  Therefore, $\Phi \vdash \alpha$.
\end{proof}

\end{document}